\documentclass[12pt]{article}
\usepackage[left=2cm,right=2cm, top=2.4cm,bottom=2.4cm,bindingoffset=0cm]{geometry}

\usepackage{tikz, titlesec, tikz-cd}
\usetikzlibrary{calc,intersections,through,backgrounds}
\usepackage{amsmath, amssymb}
\usepackage{amsthm, amsbsy}

\usepackage{amsthm}
\let\oldproofname=\proofname
\renewcommand{\proofname}{\rm\bf{\oldproofname}}
\usepackage{hyperref}

\usepackage{amsthm} 
\makeatletter
\renewenvironment{proof}[1][\proofname]{%
  \par\pushQED{\qed}\normalfont%
  \topsep6\p@\@plus6\p@\relax
  \trivlist\item[\hskip\labelsep\bfseries#1\@addpunct{.}]%
  \ignorespaces
}{%
  \popQED\endtrivlist\@endpefalse
}
\makeatother

\usepackage{mathtools}
\usepackage{placeins}

\usepackage[numbers, sort&compress]{natbib}

\usepackage{caption}
\usepackage{subcaption}

\usepackage{yhmath}
\usepackage{mathdots}

\usepackage[width=0.75\textwidth]{caption}

\usetikzlibrary{positioning, bending, arrows.meta}
\usetikzlibrary{decorations.text}
\usetikzlibrary{decorations.pathmorphing}
\usetikzlibrary{decorations.markings}

  \tikzset{->-/.style={decoration={
  markings,
  mark=at position .5 with {\arrow{>}}},postaction={decorate}}}
    \tikzset{-<-/.style={decoration={
  markings,
  mark=at position .5 with {\arrow{<}}},postaction={decorate}}}

    \tikzset{-->-/.style={decoration={
  markings,
  mark=at position .8 with {\arrow{>}}},postaction={decorate}}}

\newtheorem{lemma}{Lemma}[section]
\newtheorem{proposition}[lemma]{Proposition}
\newtheorem{theorem}[lemma]{Theorem}

\theoremstyle{definition}
\newtheorem{remark}[lemma]{Remark}
\newtheorem{definition}[lemma]{Definition}
\newtheorem{question}[lemma]{Question}
\newtheorem{problem}[lemma]{Problem}
\newtheorem{conjecture}[lemma]{Conjecture}

\newtheorem{example}[lemma]{Example}

\newcommand{\Z}{\mathbb{Z}}
\newcommand{\C}{\mathbb{C}}
\newcommand{\R}{\mathbb{R}}

\renewcommand{\P}{\mathbb{P}}

\newcommand{\pp}{\mathfrak{P}}

\begin{document}

\tikzset{->-/.style={decoration={
  markings,
  mark=at position .7 with {\arrow{>}}},postaction={decorate}}}
  
  \tikzset{->>-/.style={decoration={
  markings,
  mark=at position .5 with {\arrow{>}}},postaction={decorate}}}
  
  \tikzset{-<-/.style={decoration={
  markings,
  mark=at position .3 with {\arrow{<}}},postaction={decorate}}}
  
\usetikzlibrary{angles, quotes}

\title{The dimer model and dynamical incidence geometry}
\author{Anton Izosimov\thanks{
School of Mathematics and Statistics,
University of Glasgow;
e-mail: {\tt anton.izosimov@glasgow.ac.uk}
}\, and Pavlo Pylyavskyy\thanks{
Department of Mathematics,
University of Minnesota;
e-mail: {\tt ppylyavs@umn.edu }
} }
\date{}

\maketitle

\abstract{
We propose a geometric counterpart of the dimer model on bipartite graphs. A state of our model consists of a choice of a point for each white vertex and hyperplane for each black vertex.
This data is subject to certain conditions determined by the graph; the resulting configurations are called coherent double circuit configurations. We show that  our model behaves consistently under standard local moves of the dimer model. 
On the geometric side, this gives rise to a new class of theorems  in linear incidence geometry -- dynamical incidence theorems. 
Examples include results on pentagram maps, pentagram spirals, and Q-nets. 
We also examine the problem of parametrizing coherent double circuit configurations. In particular, we study whether, once the white-vertex part of the data is fixed, one can recover the black-vertex data from a point on the spectral curve.
}

\tableofcontents

\section{Introduction}

Incidence theorems are statements about points, lines, and possibly higher-dimensional subspaces and their incidences. Examples include classical theorems of Desargues, Pappus, and M\"obius. In this paper, we establish a connection between incidence geometry and an archetypal model of statistical physics -- the dimer model. 

The dimer model is the study of perfect matchings on bipartite graphs.  In addition to being one of the most fundamental exactly  solvable models of statistical mechanics, the dimer model arises in multiple mathematical contexts, in particular in gauge theory and integrable systems. In the latter direction, Goncharov and Kenyon showed that certain sequences of local transformations, or \emph{moves}, of the dimer model on a torus can be viewed as discrete completely integrable systems \cite{GK}. These integrable systems can also be described as sequences of mutations in an appropriate cluster algebra and are, therefore, known as \emph{cluster integrable systems}.

The construction we propose in the present paper can be viewed as a geometrization of the dimer model and Goncharov-Kenyon dynamics.  Fix a bipartite graph $\Gamma$ on an orientable surface, typically a torus, and an integer $d \geq 2$.  A state of our model consists of a choice of a point in the projective space $\P^d$ for each white vertex of $\Gamma$ and a hyperplane in $\P^d$ for each black vertex of $\Gamma$. These points and hyperplanes are subject to two conditions:

\begin{enumerate}
\item[(V)] The points/hyperplanes assigned to the neighbors of any vertex form a circuit (i.e., a minimally linear dependent set). 
\item[(F)] The multi-ratio of the collection of points/hyperplanes assigned to the boundary of every face is equal to one. 
\end{enumerate}

An assignment of points to white vertices such that the neighbors of each black vertex form a circuit is called in \cite{affolter2024vector} a \emph{circuit configuration}. 
Thus, our condition (V) can be viewed as an assignment of a circuit configuration to white vertices and a circuit configuration in the dual space to black vertices. Condition (F) is a compatibility condition between the two circuit configurations.  Configurations satisfying this condition were shown in \cite{fomin2023incidences} to encode numerous classical incidence theorems. 
Our construction can thus be viewed as a unification of models of  \cite{affolter2024vector} and \cite{fomin2023incidences}.

Our main result is that the conditions (V) and (F) are preserved by local moves of the dimer model on the graph $\Gamma$. For each particular graph, this can be interpreted as a \emph{dynamical incidence theorem}. Every such theorem concerns an intricate incidence relation between a collection of points and hyperplanes, and states that this relation is preserved by a certain geometric construction. 

To give the reader a feel for the types of theorems obtained on this path, we consider an example related to Schwartz's \emph{pentagram map} \cite{Sch}. The definition of the pentagram map is shown in Figure \ref{fig:PM}: the image of
a polygon $P$ under the pentagram map is the polygon $T(P)$ whose vertices are the intersection points of consecutive shortest diagonals of $P$, i.e. diagonals connecting second-nearest
vertices. 


For two planar polygons $P$, $Q$ with the same number of vertices, say that $Q$ is inscribed in $P$ if consecutive vertices of $Q$ lie on consecutive sides of $P$. Note that the property of being inscribed is not preserved by the pentagram map $T$: if $Q$ is inscribed in $P$, then it is not necessarily the case that $T(Q)$ is inscribed in $T(P)$. However, it turns out that if this does hold, then this also holds for  the next, and hence all,  iterations of the pentagram map:


\begin{figure}[t]
\centering
\includegraphics[scale = 0.25]{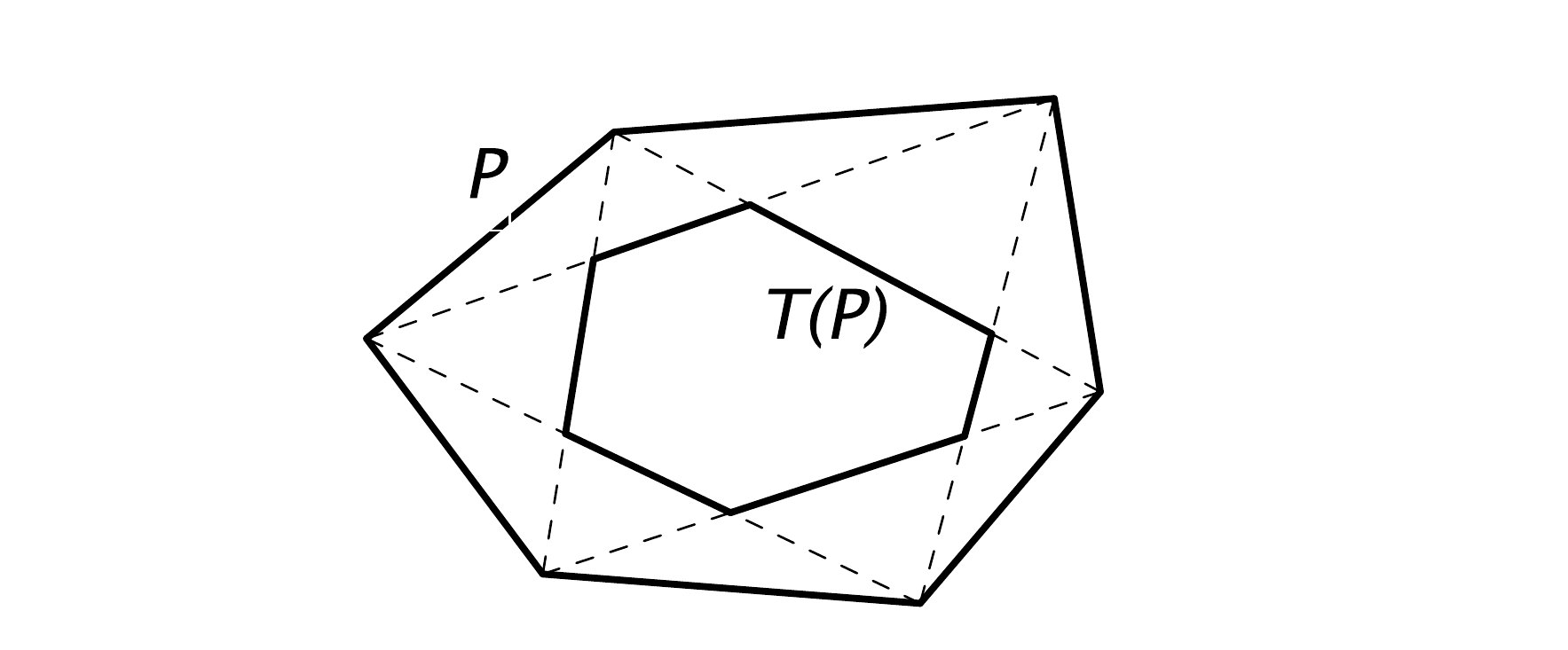}

%
%
%
%
%
%
\caption{The pentagram map.}\label{fig:PM}
\end{figure}
\begin{figure}[b!]
\centering
  \centering
\includegraphics[scale = 0.35]{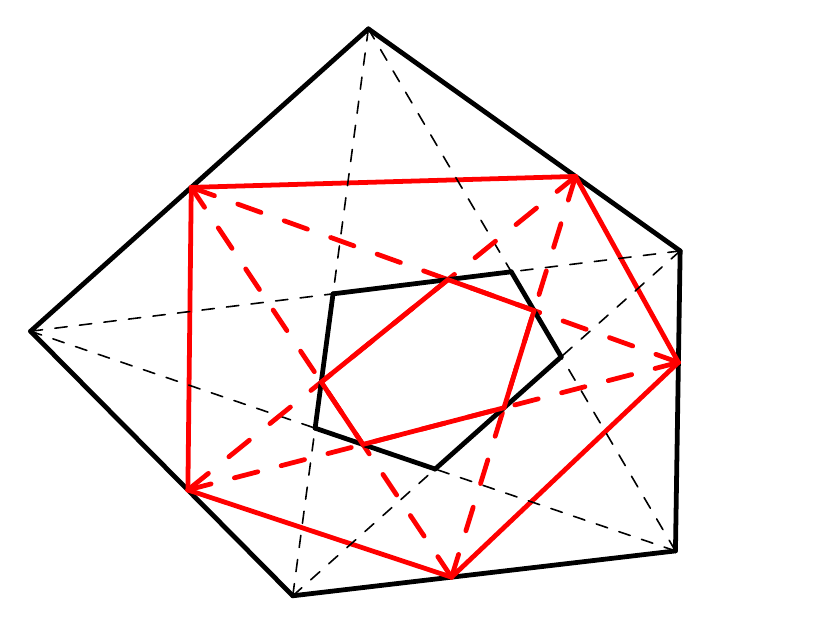}

\caption{Polygons inscribed in each other whose images under the pentagram map are also inscribed in each other.
}\label{fig:PM2}
\end{figure}
\begin{theorem}\label{thm1}
Let $T$ be the pentagram map, and $P$, $Q$ be planar polygons such that $Q$ is inscribed in $P$, and $T(Q)$ in $T(P)$, see Figure \ref{fig:PM2}. Then $T^2(Q)$ is also inscribed in $T^2(P)$, see Figure \ref{fig:PM3}.
\end{theorem}

\begin{figure}[t]
\centering
  \centering
\includegraphics[scale = 0.6]{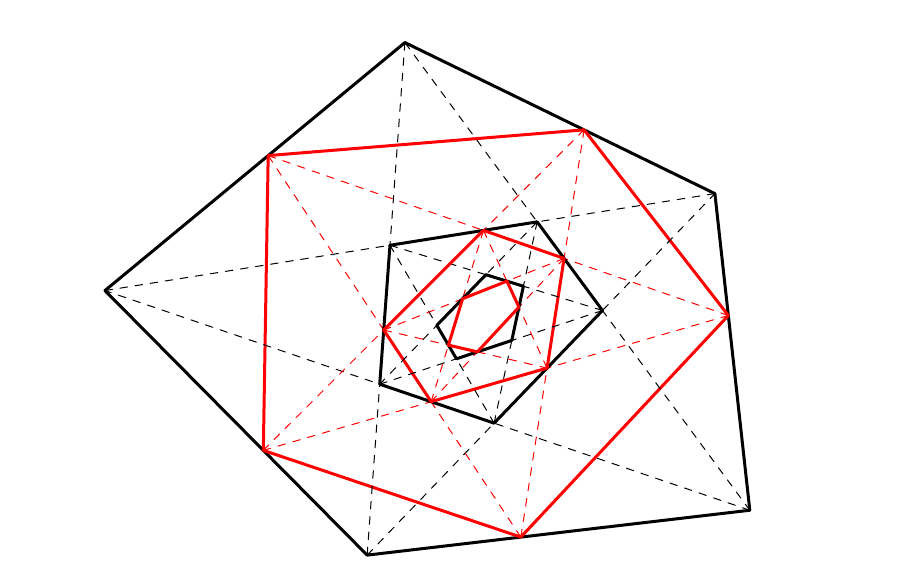}

\caption{If $Q$ is inscribed in $P$, and $T(Q)$ in $T(P)$,  then $T^2(Q)$ is also inscribed in $T^2(P)$.
}\label{fig:PM3}
\end{figure}

Note that this theorem can be applied again to $T(P)$ and $T(Q)$, giving that $T^3(Q)$ is inscribed in $T^3(P)$. Analogously, one gets that $T_k^m(Q)$ is  inscribed in $T_k^m(P)$ for all $m \in \Z$. We prove this result in Section \ref{sec:ex1}. To that end, we show that, for an appropriate bipartite graph on a torus, point-line configurations satisfying the above conditions (V) and (F) are in bijection with pairs of polygons $P$, $Q$ such that $Q$ is inscribed in $P$ and $T(Q)$ in $T(P)$. Furthermore, we provide a sequence of local moves which preserves the combinatorics of the graph but replaces polygons $P$ and $Q$ with, respectively, $T(P)$ and $T(Q)$. So, since the combinatorics of the graph remains the same, and the conditions (V) and (F) survive under local moves, it follows that  $T(P)$ and $T(Q)$ must satisfy the same conditions as $P$ and $Q$. In particular, $T(T(Q))$ must also be inscribed in $T(T(P))$, q.e.d.

More generally, each time we have a bipartite graph and a sequence of moves which preserves its combinatorics, we get an incidence theorem such as Theorem \ref{thm1}. The input of every such theorem is a certain non-trivial incidence, such as the one in Figure \ref{fig:PM2}. The theorem says that applying a certain geometric construction produces an incidence with the same combinatorics. 

We note that Theorem \ref{thm1}, just like all other theorems stemming from our model, can be proved in an elementary way, by a number of successive applications of Desargues' theorem. A non-trivial problem is \emph{finding} theorems of this type. Our model produces such a theorem for any bipartite graph.

Another problem we discuss in the paper is the existence and parametrization of incidences which serve as inputs of our theorems. For example, how does one construct polygons $P$ and $Q$ such that  $Q$ is inscribed in $P$ and $T(Q)$ in $T(P)$ (where $T$ is the pentagram map)?  In Section~\ref{sec:rec} we prove that, given $P$, there exists a one-dimensional family of polygons $Q$ satisfying these conditions. This family is parametrized by an algebraic curve, known as the \emph{spectral curve}. For example, for generic pentagons the spectral curve is elliptic. This means that  if we fix the outer  pentagon in Figure \ref{fig:PM2}, then the set of configurations as shown forms an elliptic curve. More generally, for an arbitrary bipartite graph on a torus, we discuss the problem of recovering hyperplanes attached to black vertices from points attached to white vertices. We show that this problem is closely related to the spectral curve of the corresponding dimer model.

\begin{remark}
We note that while {parametrization} of pairs of polygons $P$ and $Q$ such that  $Q$ is inscribed in $P$ and $T(Q)$ in $T(P)$ turns out to be a rather difficult problem, an {example} of such a pair can be constructed as follows. Suppose $P$ is a generic pentagon. Then there is a unique conic tangent to all sides of $P$. Denote the polygon formed by the tangency points by $I(P)$. Then, by a theorem of Kasner  \cite{kasner1928projective}, we have $T(I(P)) = I(T(P))$. So, setting $Q = I(P)$, we get, for every pentagon $P$, a pair of polygons $P$ and $Q$ such that  $Q$ is inscribed in $P$ and $T(Q)$ in $T(P)$. 

More generally, this construction works for \emph{Poncelet polygons}, i.e., polygons inscribed in a conic and circumscribed about another conic. Indeed, for such polygons, one still has $T(I(P)) = I(T(P))$ \cite{tabachnikov2019kasner}.

Even more generally, suppose $P$ is a polygon {circumscribed} about a conic (but not necessarily inscribed). As before, let $Q = I(P)$ be  the polygon inscribed in $P$ formed by the points where the inscribed conic touches the sides of $P$. 
In that case, it may not be true that $T(Q) = I(T(P))$. Indeed, $I(T(P))$ may not even be defined, because the property of being circumscribed is not preserved by the pentagram map. However, it is still true that $T(Q)$ is inscribed in $T(P)$, see Proposition \ref{thm:ins} below.
\end{remark}

\bigskip

{\bf Acknowledgments.} A.I. is grateful to Max Planck Institute for Mathematics in
Bonn for its hospitality and financial support. A.I. was partially supported by the Simons Foundation through its Travel Support for Mathematicians program. 
P.P. was partially supported by the NSF DMS-1949896 and by Simons Foundation through its Simons Fellows in
Mathematics program and Travel Support for Mathematicians program.

\section{The dimer model and Gocharov-Kenyon dynamics}\label{sec:gk}
In this section, we review aspects of the dimer model relevant for the present paper. For details, see \cite{GK}.

Given a graph, its \textit{dimer cover}, or a \textit{perfect matching}, is a set of edges with the property that every vertex is adjacent to a unique edge of the cover. 
The study of perfect matchings and their statistics on a given graph is known as the \emph{dimer model}. In this paper, we are only interested in the dimer model on \emph{bipartite graphs}, i.e. graphs whose vertices are colored black and white in such a way that no edge joins two vertices of the same color. Furthermore, we will always assume that our graph $\Gamma$ is embedded in a $2$-dimensional surface $\Sigma$ in such a way that its \textit{faces}, i.e.  connected components of the complement $\Sigma \setminus \Gamma$, are contractible.

The probability of each particular dimer cover in the dimer model is defined using numbers, or \emph{weights}, written on edges or faces of the graph. In the probabilistic context, the weights are assumed to be positive and real, but in more general situations it is common to allow arbitrary complex weights. We will not be explicitly using weights in the present paper, although they will always be implicitly present and may be calculated based on our geometric data. 

A \emph{move} is a local transformation of a bipartite graph which is, in a certain sense, an automorphism of the dimer model. Each such move affects, albeit only locally, both the combinatorics of the graph and the weights. Since transformations of weights are irrelevant for our purposes, we only describe the change in combinatorics.

The first move type is \emph{degree two vertex removal} shown in Figure \ref{shrink}. The figure depicts removal of a degree two white vertex, but the same move is allowed with all colors
reversed. The neighbors of a vertex being removed can have any degree. The inverse transformation is also allowed and is called \emph{degree two vertex addition}.

 \begin{figure}[t]
 \centering
\begin{tikzpicture}[, scale = 0.8]
\node [draw,circle,color=black, fill=black,inner sep=0pt,minimum size=5pt] (A) at (0,0) {};
\node [draw,circle,color=black, fill=white,inner sep=0pt,minimum size=5pt] (B) at (1,0) {};
\node [draw,circle,color=black, fill=black,inner sep=0pt,minimum size=5pt] (C) at (2,0) {};
\node [draw,circle,color=black, fill=black,inner sep=0pt,minimum size=5pt] (D) at (7,0) {};
%
%
%
\draw (A) -- (B) -- (C);
\draw (A) -- +(-0.5,+0.5);
\draw (A) -- +(-0.7,+0.);
\draw (A) -- +(-0.5,-0.5);
\draw (C) -- +(+0.5,+0.5);
\draw (C) -- +(+0.7,+0.);
\draw (C) -- +(+0.5,-0.5);
\node () at (4.5,0) {$\longleftrightarrow$};
\draw (D) -- +(-0.5,+0.5);
\draw (D) -- +(-0.5,-0.5);
\draw (D) -- +(+0.5,+0.5);
\draw (D) -- +(+0.5,-0.5);
\draw (D) -- +(+0.7,+0.);
\draw (D) -- +(-0.7,+0.);
\end{tikzpicture}
\caption{A degree two vertex removal/addition.}\label{shrink}
\end{figure}
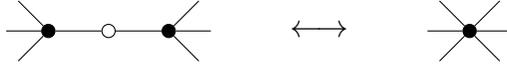

%
%

 \begin{figure}[b]
 \centering
\begin{tikzpicture}[scale = 0.5]
\node [draw,circle,color=black, fill=white,inner sep=0pt,minimum size=5pt] (A) at (-9.5,0) {};
\draw (A) -- +(-1,0);
\draw (A) -- +(0,-1);
\draw (A) -- +(-.7,-.7);
\node [draw,circle,color=black, fill=black,inner sep=0pt,minimum size=5pt] (B) at (-6.5,0) {};
\draw (B) -- +(1,0);
\draw (B) -- +(0,-1);
\draw (B) -- +(.7,-.7);
\node [draw,circle,color=black, fill=white,inner sep=0pt,minimum size=5pt] (C) at (-6.5,3) {};
\draw (C) -- +(1,0);
\draw (C) -- +(0,1);
\draw (C) -- +(.7,.7);
\node [draw,circle,color=black, fill=black,inner sep=0pt,minimum size=5pt] (D) at (-9.5,3) {};
\draw (D) -- +(-1,0);
\draw (D) -- +(0,1);
\draw (D) -- +(-.7,.7);
\draw [thin] (A) -- (B) node[midway, below] {$$}; \draw (B) -- (C) node[midway, right] {$$}; \draw (C) -- (D) node[midway, above] {$$}; \draw (D) -- (A) node[midway, left] {$$};;
\node [draw,circle,color=black, fill=white,inner sep=0pt,minimum size=5pt] (A) at (0,0) {};
\node [draw,circle,color=black, fill=black,inner sep=0pt,minimum size=5pt] (B) at (3,0) {};
\node [draw,circle,color=black, fill=white,inner sep=0pt,minimum size=5pt] (C) at (3,3) {};
\node [draw,circle,color=black, fill=black,inner sep=0pt,minimum size=5pt] (D) at (0,3) {};
\node [draw,circle,color=black, fill=black,inner sep=0pt,minimum size=5pt] (A1) at (0.7,0.7) {};
\node [draw,circle,color=black, fill=white,inner sep=0pt,minimum size=5pt] (B1) at (2.3,0.7) {};
\node [draw,circle,color=black, fill=black,inner sep=0pt,minimum size=5pt] (C1) at (2.3,2.3) {};
\node [draw,circle,color=black, fill=white,inner sep=0pt,minimum size=5pt] (D1) at (0.7,2.3) {};
\draw [thin] (A1) -- (B1) node[midway, below] {$$}; \draw (B1) -- (C1) node[midway, right] {$$}; \draw (C1) -- (D1) node[midway, above] {$$}; \draw (D1) -- (A1) node[midway, left] {$$};;
\draw (A) -- (A1);
\draw (B) -- (B1);
\draw (C) -- (C1);
\draw (D) -- (D1);
\draw (A) -- +(-1,0);
\draw (A) -- +(0,-1);
\draw (A) -- +(-.7,-.7);
\draw (B) -- +(1,0);
\draw (B) -- +(0,-1);
\draw (B) -- +(.7,-.7);
\draw (C) -- +(1,0);
\draw (C) -- +(0,1);
\draw (C) -- +(.7,.7);
\draw (D) -- +(-1,0);
\draw (D) -- +(0,1);
\draw (D) -- +(-.7,.7);
\draw [-> ](-4,1.5) -- (-3, 1.5);
\end{tikzpicture}
\caption{An urban renewal. }\label{urban}
\end{figure}
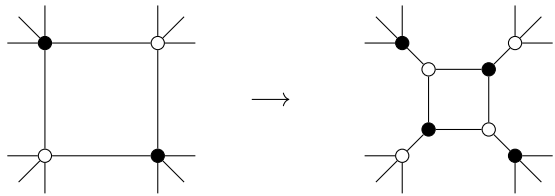

The second move type is the \emph{urban renewal}, see Figure~\ref{urban}. The vertices surrounding the square face can have any degee. There is also a simplified version of the urban renewal known as the \emph{spider move}; up to vertex removal/addition, the urban renewal and spider move are the same.  
The inverse of the urban renewal is not considered since that move is an involution, up to degree two vertex removals. Although we do not discuss transformations of weights, we note that in terms of face weights an urban renewal amounts to a transformation known as \emph{cluster mutation}. This observation underlies the connection between the dimer model and cluster algebras.\par
%
%

In \cite{GK}, Goncharov and Kenyon showed how to use the dimer model on a torus to build what they called \emph{cluster integrable systems}. Such a system arises every time we have a sequence of moves which restores the combinatorics of the initial graph. Although the combinatorics remains the same, the weights undergo a certain birational transformation. For so-called \emph{minimal} graphs, that birational map is discrete integrable in the sense that it preserves a Poisson structure and a maximal family of Poisson-commuting Hamiltonians. The latter are extracted from the so-called \emph{partition function}, which governs statistics of dimer covers of the graph.

\section{Circuit configurations}\label{sec:crc}
Here we recall the notion of a circuit configuration introduced in \cite{affolter2024vector}. 

Recall that a circuit in a projective space is a linearly dependent set of points whose any proper subset is linearly independent. Two points form a circuit when they coincide, three points form a circuit when they are collinear and pairwise distinct, etc.

\begin{definition}\label{def:cc}Let $\Gamma$  be a bipartite graph, and $d \geq 2$ be an integer. A \emph{$d$-dimensional circuit configuration} on $\Gamma$ is an assignment of a point in the projective space $\P^d$ to every white vertex of $\Gamma$ such that the points assigned to the neighbors of any black vertex form a circuit. \end{definition}

This definition also admits an affine version where points in the projective space are replaced by vectors in a vector space. Such configurations are called \emph{vector-relation configurations} \cite{affolter2024vector}. In the present paper, we are only interested in the projective setting.


 \begin{figure}[t]
 \centering
\begin{tikzpicture}[, scale = 1]
\node [draw,circle,color=black, fill=black,inner sep=0pt,minimum size=5pt] (A) at (0,0) {};
\node [draw,circle,color=black, fill=white,,inner sep=0pt,minimum size=5pt,label = left:\footnotesize${B_1}$] (ANW) at (-0.5,0.5) {};
\node [draw,circle,color=black, fill=white,,inner sep=0pt,minimum size=5pt,label = left:\footnotesize$B_k$] (ASW) at (-0.5,-0.5) {};
\node [label = {[label distance=-0cm]left:\footnotesize${\vdots}$}] () at (-0.6,0.1) {};
\node [draw,circle,color=black, fill=white,inner sep=0pt,minimum size=5pt, label = \footnotesize$A$] (B) at (1,0) {};
\node [draw,circle,color=black, fill=black,inner sep=0pt,minimum size=5pt] (C) at (2,0) {};
\node [draw,circle,color=black, fill=white,,inner sep=0pt,minimum size=5pt,label = right:\footnotesize$C_1$] (CNE) at (2.5,0.5) {};
\node [draw,circle,color=black, fill=white,,inner sep=0pt,minimum size=5pt,label = right:\footnotesize$C_l$] (CSE) at (2.5,-0.5) {};
\node [label = {[label distance=-0cm]right:\footnotesize$\vdots$}] () at (2.6,0.1) {};
\node [draw,circle,color=black, fill=black,inner sep=0pt,minimum size=5pt] (D) at (7,0) {};
\node [draw,circle,color=black, fill=white,,inner sep=0pt,minimum size=5pt,label = left:\footnotesize$B_1$] (DNW) at (6.5,0.5) {};
\node [draw,circle,color=black, fill=white,,inner sep=0pt,minimum size=5pt,label =left:\footnotesize$B_k$] (DSW) at (6.5,-0.5) {};
\node [label = {[label distance=-0cm]left:\footnotesize$\vdots$}] () at (6.4,0.1) {};
\node [draw,circle,color=black, fill=white,,inner sep=0pt,minimum size=5pt,label = right:\footnotesize$C_1$] (DSE) at (7.5,0.5) {};
\node [draw,circle,color=black, fill=white,,inner sep=0pt,minimum size=5pt,label = right:\footnotesize$C_l$] (DNE) at (7.5,-0.5) {};
\node [label = {[label distance=-0cm]right:\footnotesize$\vdots$}] () at (7.6,0.1) {};
\draw (A) -- (B) -- (C);
\draw (A) -- (ANW);
\draw (A) -- (ASW);
\draw (C) -- (CNE);
\draw (C) -- (CSE);
\node () at (4.5,0) {$\longleftrightarrow$};
\draw (D) -- (DSW);
\draw (D) -- (DNW);
\draw (D) -- (DSE);
\draw (D) -- (DNE);

\end{tikzpicture}
\caption{Removal/addition of a degree two white vertex in a circuit configuration.}\label{shrink2}
\end{figure}
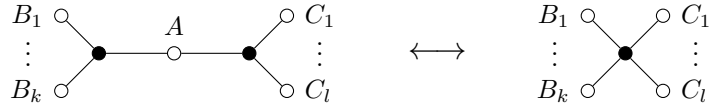

%
%
%
%
 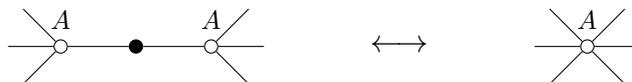
\begin{figure}[b]
 \centering
\begin{tikzpicture}[, scale = 1, ,]
\node [draw,circle,color=black, fill=white,inner sep=0pt,minimum size=5pt, label = \footnotesize$A$] (A) at (0,0) {};
\node [draw,circle,color=black, fill=black,inner sep=0pt,minimum size=5pt] (B) at (1,0) {};
\node [draw,circle,color=black, fill=white,inner sep=0pt,minimum size=5pt, label = \footnotesize$A$] (C) at (2,0) {};
\node [draw,circle,color=black, fill=white,inner sep=0pt,minimum size=5pt, label = \footnotesize$A$] (D) at (7,0) {};
%
%
%
\draw (A) -- (B) -- (C);
\draw (A) -- +(-0.5,+0.5);
\draw (A) -- +(-0.7,+0.);
\draw (A) -- +(-0.5,-0.5);
\draw (C) -- +(+0.5,+0.5);
\draw (C) -- +(+0.7,+0.);
\draw (C) -- +(+0.5,-0.5);
\node () at (4.5,0) {$\longleftrightarrow$};
\draw (D) -- +(-0.5,+0.5);
\draw (D) -- +(-0.5,-0.5);
\draw (D) -- +(+0.5,+0.5);
\draw (D) -- +(+0.5,-0.5);
\draw (D) -- +(+0.7,+0.);
\draw (D) -- +(-0.7,+0.);

\end{tikzpicture}
\caption{Removal/addition of a degree two black vertex in a circuit configuration.}\label{shrink3}
\end{figure}
%
%
%
%
%
%
 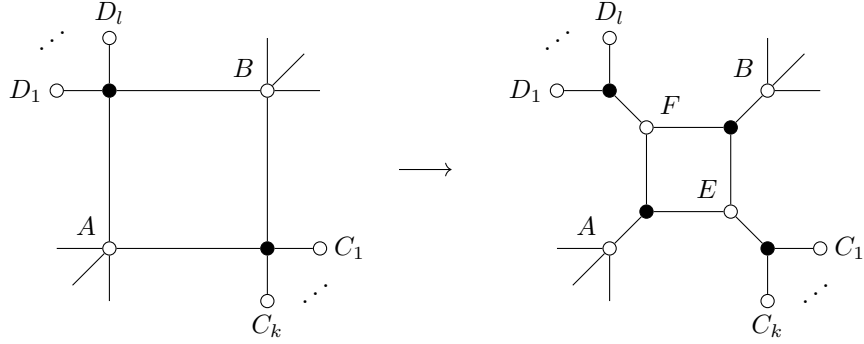
\begin{figure}[t]
 \centering
\begin{tikzpicture}[scale = 0.7, , label distance=-0.5mm]
\node [draw,circle,color=black, fill=white,inner sep=0pt,minimum size=5pt, label = 135:\footnotesize$A$] (A) at (-9.5,0) {};
\draw (A) -- +(-1,0);
\draw (A) -- +(0,-1);
\draw (A) -- +(-.7,-.7);
\node [draw,circle,color=black, fill=black,inner sep=0pt,minimum size=5pt] (B) at (-6.5,0) {};
\node [draw,circle,color=black, fill=white,inner sep=0pt,minimum size=5pt, label = right:\footnotesize$C_1$] (B1) at (-5.5,0) {};
\node [draw,circle,color=black, fill=white,inner sep=0pt,minimum size=5pt, label = below:\footnotesize$C_k$] (B2) at (-6.5,-1) {};
\node [] () at (-5.6,-0.7) {\footnotesize$\iddots$};
\draw (B) -- (B1);
\draw (B) -- (B2);
\node [draw,circle,color=black, fill=white,inner sep=0pt,minimum size=5pt, label = 135:\footnotesize$B$] (C) at (-6.5,3) {};
\draw (C) -- +(1,0);
\draw (C) -- +(0,1);
\draw (C) -- +(.7,.7);
\node [draw,circle,color=black, fill=black,inner sep=0pt,minimum size=5pt] (D) at (-9.5,3) {};
\node [draw,circle,color=black, fill=white,inner sep=0pt,minimum size=5pt, label = left:\footnotesize$D_1$] (D1) at (-10.5,3) {};
\node [draw,circle,color=black, fill=white,inner sep=0pt,minimum size=5pt, label = above:\footnotesize$D_l$] (D2) at (-9.5,4) {};
\node [] () at (-10.6,4.1) {\footnotesize$\iddots$};
\draw (D) -- (D1);
\draw (D) -- (D2);
\draw [thin] (A) -- (B) node[midway, below] {$$}; \draw (B) -- (C) node[midway, right] {$$}; \draw (C) -- (D) node[midway, above] {$$}; \draw (D) -- (A) node[midway, left] {$$};;
\node [draw,circle,color=black, fill=white,inner sep=0pt,minimum size=5pt, label = 135:\footnotesize$A$] (A) at (0,0) {};
\node [draw,circle,color=black, fill=black,inner sep=0pt,minimum size=5pt] (B) at (3,0) {};
\node [draw,circle,color=black, fill=white,inner sep=0pt,minimum size=5pt, label = 135:\footnotesize$B$] (C) at (3,3) {};
\node [draw,circle,color=black, fill=black,inner sep=0pt,minimum size=5pt] (D) at (0,3) {};
\node [draw,circle,color=black, fill=black,inner sep=0pt,minimum size=5pt] (A1) at (0.7,0.7) {};
\node [draw,circle,color=black, fill=white,inner sep=0pt,minimum size=5pt, label = 135:\footnotesize$\,E$] (B1) at (2.3,0.7) {};
\node [draw,circle,color=black, fill=black,inner sep=0pt,minimum size=5pt] (C1) at (2.3,2.3) {};
\node [draw,circle,color=black, fill=white,inner sep=0pt,minimum size=5pt, label =45:\footnotesize$F$] (D1) at (0.7,2.3) {};
\draw [thin] (A1) -- (B1) node[midway, below] {$$}; \draw (B1) -- (C1) node[midway, right] {$$}; \draw (C1) -- (D1) node[midway, above] {$$}; \draw (D1) -- (A1) node[midway, left] {$$};;
\draw (A) -- (A1);
\draw (B) -- (B1);
\draw (C) -- (C1);
\draw (D) -- (D1);
\draw (A) -- +(-1,0);
\draw (A) -- +(0,-1);
\draw (A) -- +(-.7,-.7);
\node [draw,circle,color=black, fill=white,inner sep=0pt,minimum size=5pt, label = right:\footnotesize$C_1$] (B1) at (4,0) {};
\node [draw,circle,color=black, fill=white,inner sep=0pt,minimum size=5pt, label = below:\footnotesize$C_k$] (B2) at (3,-1) {};
\node [] () at (3.9,-0.7) {\footnotesize$\iddots$};
\draw (B) -- (B1);
\draw (B) -- (B2);
\draw (C) -- +(1,0);
\draw (C) -- +(0,1);
\draw (C) -- +(.7,.7);
\node [draw,circle,color=black, fill=white,inner sep=0pt,minimum size=5pt, label = left:\footnotesize$D_1$] (D1) at (-1,3) {};
\node [draw,circle,color=black, fill=white,inner sep=0pt,minimum size=5pt, label = above:\footnotesize$D_l$] (D2) at (0,4) {};
\node [] () at (-1,4.1) {\footnotesize$\iddots$};
\draw (D) -- (D1);
\draw (D) -- (D2);
\draw [-> ](-4,1.5) -- (-3, 1.5);

\end{tikzpicture}
\caption{An urban renewal of a circuit configuration. Here $E= \langle A, B\rangle \cap \langle C_1, \dots, C_k \rangle$, $F= \langle A, B\rangle \cap \langle D_1, \dots, D_l \rangle.$}\label{urban2}
\end{figure}

Circuit configurations can be understood as a geometric realization of the dimer model. In particular, circuit configurations on graphs embedded in a surface admit the same kind of moves as the dimer model, see Figures \ref{shrink2},~\ref{shrink3},~\ref{urban2}. Next to each white vertex, we write the associated point in $\P^d$. In each case, it is readily verified that for generic input the circuit conditions on the initial graph imply the circuit conditions on the resulting one. The only move that shouldn't be allowed is white vertex removal which creates a black vertex of degree exceeding $d+2$, as such a vertex would have too many neighbors for the associated points to form a circuit. 
%



The connection between the dimer model moves and moves on circuit configurations goes beyond the same combinatorics. To every circuit configuration one can assign certain cross-ratio-type numerical invariants, one for every face of the graph (here we assume that the graph is embedded in a surface, so that the faces are well-defined). Then, as the configuration transforms as shown in the figures, its numerical invariants undergo the same transformation as the face weights in the dimer model. Thus, one can indeed think of circuit configurations and the corresponding moves as geometric counterparts of the dimer model and Goncharov-Kenyon dynamics.


\section{Coherence and the master theorem}\label{sec:coh}

Here we review a different geometric model on bipartite graphs -- coherent tilings. It was introduced in \cite{fomin2023incidences} to describe a common framework for multiple classical incidence theorems.

\begin{definition}\label{def:tiling} Let $\Sigma$ be a closed surface. A \emph{tiling} of $\Sigma$ is a bipartite graph which partitions $\Sigma$ into topological disks, together with an assignment of a point in $\P^d$ to every white vertex and a hyperplane in $\P^d$ to every black vertex. \end{definition}
Faces of a tiling are called \emph{tiles}. Each tile is a topological $2n$-gon, whose vertices are colored white and black in the alternating fashion. White vertices are labeled with points, while black vertices are labeled with hyperplanes. 
\begin{definition}\label{def:ct} A tile is called \emph{coherent} if
the multi-ratio of points/hyperplanes assigned to its vertices taken in clockwise order is equal to $1$.
  \end{definition}

Given $n$ points $A_1, \dots, A_n$ and $n$ hyperplanes $\ell_1, \dots, \ell_n$ in $\P^d$, their multi-ratio is defined by
$$
[A_1, \ell_1, \dots, A_n, \ell_n] := \frac{\boldsymbol{\ell}_1(\mathbf{A}_1) \cdots \boldsymbol{\ell}_n(\mathbf A_n)}{\boldsymbol{\ell}_1(\mathbf{A}_2) \cdots \boldsymbol{\ell}_n(\mathbf A_1)}.
$$
Here $\mathbf A_i$ and $\boldsymbol{\ell}_i$ are, respectively, vectors and covectors in a $(d+1)$-dimensional vector space representing points $A_i$ and hyperplanes $\ell_i$ in $\P^d$. 
 Note that the condition $[A_1, \ell_1, \dots, A_n, \ell_n] = 1$ in particular implies that the point $A_i$ does not lie on hyperplanes $\ell_i, \ell_{i-1}$ (where indices are understood modulo $n$). In  \cite{fomin2023incidences} this is included in the definition of coherence. 
 
 \bigskip
 
 The coherence condition has a particularly simple meaning for quadrilaterals: 
  
  \begin{proposition}[see \cite{fomin2023incidences}, Proposition 2.5] \label{prop:coh}
  A quadrilateral tile with vertex labels $A,B, c,d$, where $A,B$ are points and  $c,d$ are hyperplanes, is coherent if and only if either $A=B$, or $c=d$, or else the line $AB$ has a non-trivial intersection with the codimension $2$ subspace $c \cap d$. 
  \end{proposition}
  
  \begin{figure}[t]
\centering
  \centering
\includegraphics[scale = 0.4]{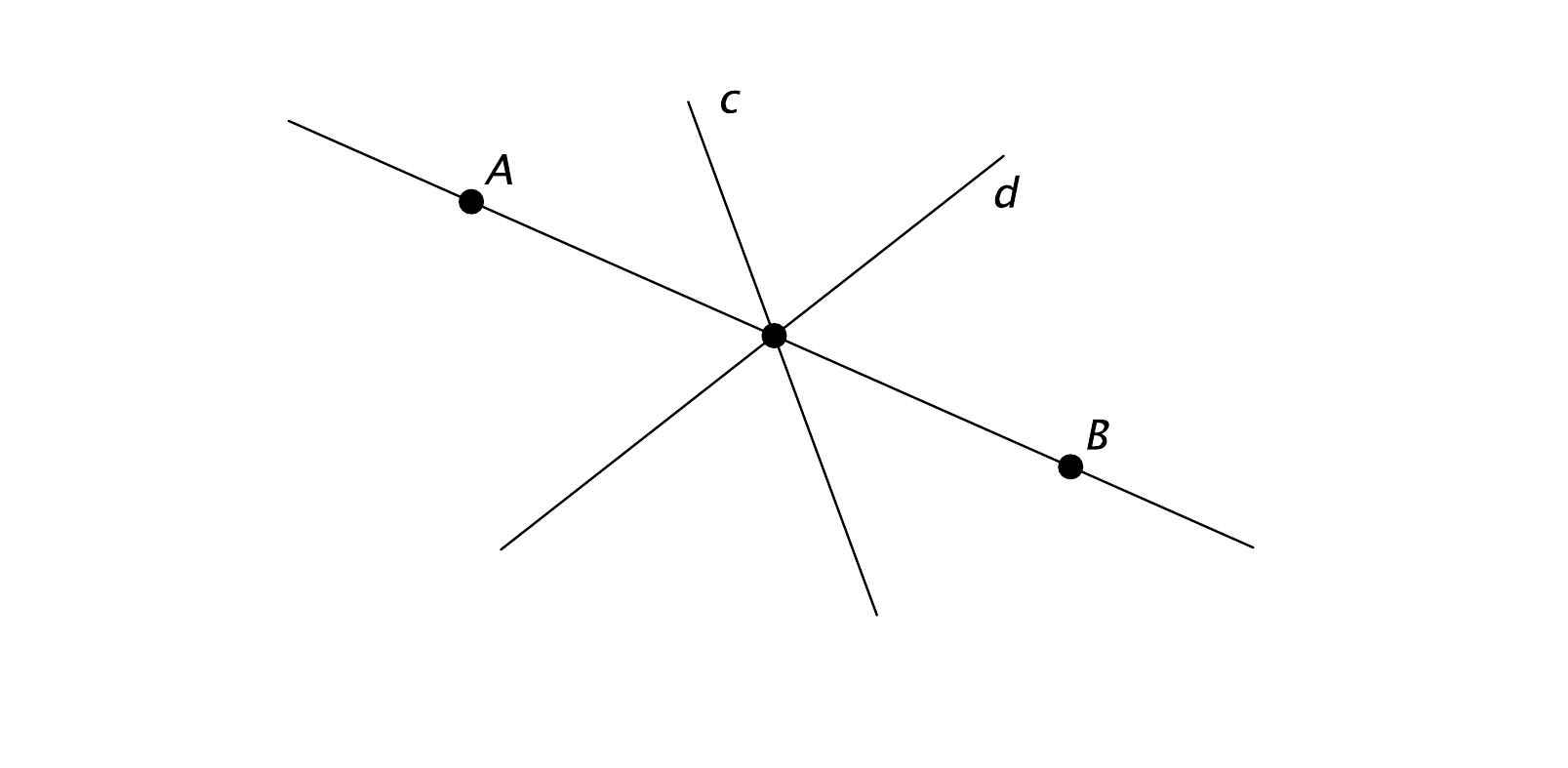}

\caption{Coherence of a quadrilateral labeled with points $A,B$ and lines $c, d$.
}\label{fig:quad}
\end{figure}
In the planar case, the latter condition just means that the lines $AB,c,d$ are concurrent, see Figure \ref{fig:quad}. 
  In two dimensions one also has the following interpretation of coherent hexagons:
    \begin{proposition}[see \cite{fomin2023incidences}, Corollary 4.4]\label{prop:des}
  Suppose $A_1B_1C_1$ and $A_2B_2C_2$ are two generic planar triangles, with sides $a_1 = B_1C_1, b_1 = A_1C_1, c_1 = A_1B_1$, $a_2 = B_2C_2, b_2 = A_2C_2, c_2 = A_2B_2$. Then a hexagon  whose vertices are clockwise labeled $A_1, b_2, C_1, a_2, B_1, c_2$ is coherent if and only if the triangles  $A_1B_1C_1$ and $A_2B_2C_2$ constitute a Desargues configuration, i.e., one of the following equivalent conditions hold:
  \begin{itemize}
  \item  the triangles  $A_1B_1C_1$ and $A_2B_2C_2$ are centrally perspective, meaning that the lines $A_1A_2$, $B_1B_2$, $C_1C_2$ are concurrent;
    \item  the triangles  $A_1B_1C_1$ and $A_2B_2C_2$ are axially perspective, meaning that the points $a_1 \cap a_2$, $b_1 \cap b_2$, $c_1 \cap c_2$ are collinear. 
  \end{itemize}
  See Figure \ref{fig:dsrg}.
  \end{proposition}
  
  \begin{figure}[t]
\centering
  \centering
\includegraphics[scale = 0.4]{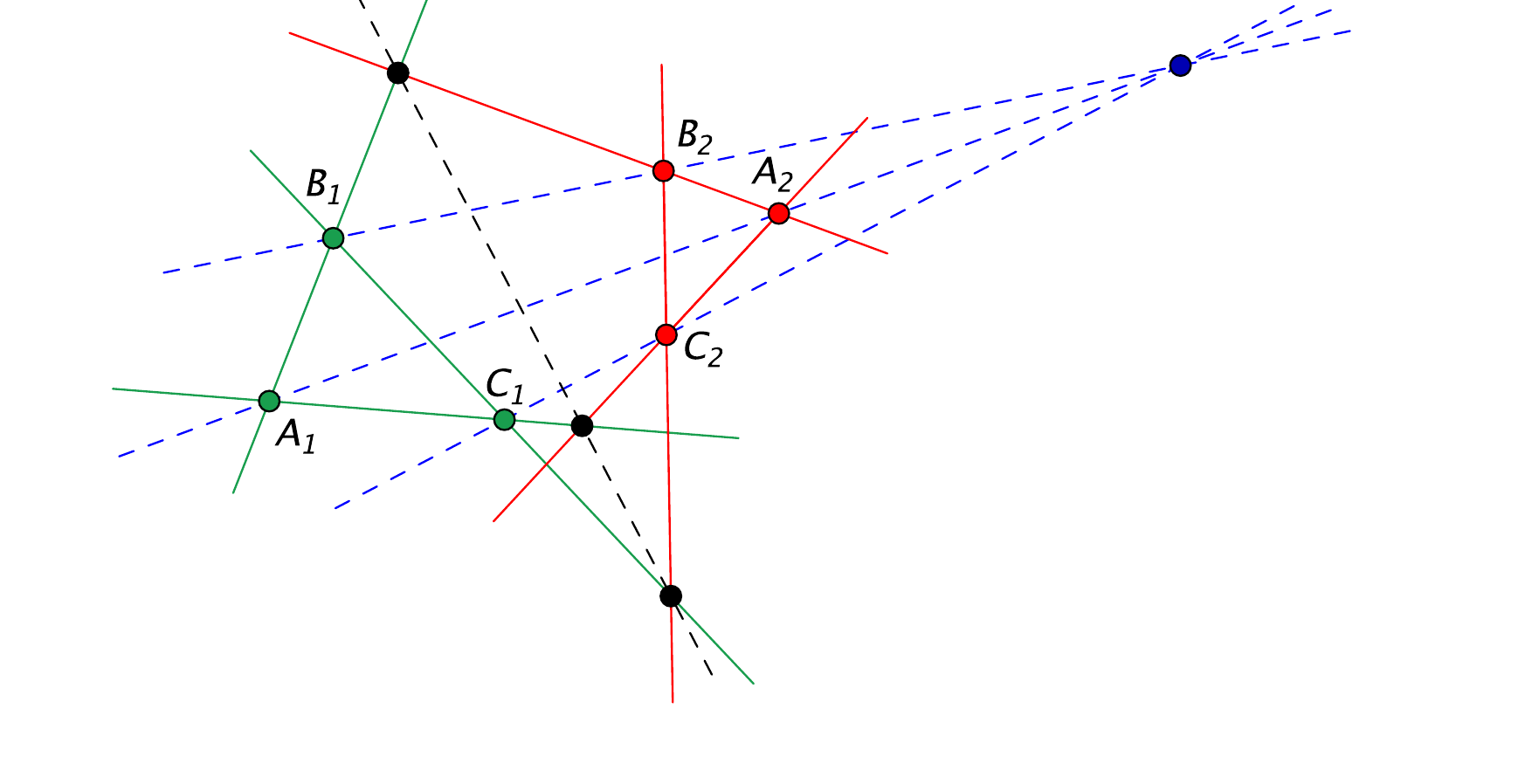}

\caption{Desargues configuration; this is equivalent to coherence of the hexagon labeled $A_1, b_2, C_1, a_2, B_1, c_2$, where $a_2 = B_2C_2, b_2 = A_2C_2, c_2 = A_2B_2$.
}\label{fig:dsrg}
\end{figure}
  
  The main application of coherence is the following \emph{master theorem}:
  
  \begin{theorem}[see \cite{fomin2023incidences}, Theorem 2.6]\label{mthm} Consider a tiling of a closed oriented surface. If all but one tiles are coherent,  then the remaining tile is coherent as well.
  
  \end{theorem}
  It is shown in  \cite{fomin2023incidences} that multiple classical results of linear incidence geometry are particular instances of this theorem.
  
\section{Coherent double circuit configurations and moves}\label{sec:main}

\begin{definition}
Let $\Sigma$ be a closed surface. A tiling of $\Sigma$ (see Definition \ref{def:tiling}) is called  a \emph{double circuit configuration} if points/hyperplanes assigned to neighbors of every vertex form a circuit.
\end{definition}
Since hyperplanes can be regarded as points in the dual projective space, the notion of a circuit of hyperplanes is well-defined. For instance, three lines in $\P^2$ for a circuit if they are pairwise distinct and concurrent. \par
Since a double circuit configuration is, basically, two circuit configurations, moves are defined in an obvious way, see Figures \ref{shrink4} and \ref{urban3}.  Next to each vertex we write the associated point or hyperplane in $\P^d$.

 \begin{figure}[t]
 \centering
\begin{tikzpicture}[, scale = 1, , label distance=-0.5mm]
\node [draw,circle,color=black, fill=black,inner sep=0pt,minimum size=5pt, label = \footnotesize$\ell$] (A) at (0,0) {};
\node [draw,circle,color=black, fill=white,,inner sep=0pt,minimum size=5pt,label = left:\footnotesize$B_1$] (ANW) at (-0.5,0.5) {};
\node [draw,circle,color=black, fill=white,,inner sep=0pt,minimum size=5pt,label = left:\footnotesize$B_k$] (ASW) at (-0.5,-0.5) {};
\node [label = {[label distance=-0cm]left:\footnotesize${\vdots}$}] () at (-0.6,0.1) {};
\node [draw,circle,color=black, fill=white,inner sep=0pt,minimum size=5pt, label = \footnotesize$A$] (B) at (1,0) {};
\node [draw,circle,color=black, fill=black,inner sep=0pt,minimum size=5pt, label = \footnotesize$\ell$] (C) at (2,0) {};
\node [draw,circle,color=black, fill=white,,inner sep=0pt,minimum size=5pt,label = right:\footnotesize$C_1$] (CNE) at (2.5,0.5) {};
\node [draw,circle,color=black, fill=white,,inner sep=0pt,minimum size=5pt,label = right:\footnotesize$C_l$] (CSE) at (2.5,-0.5) {};
\node [label = {[label distance=-0cm]right:\footnotesize$\vdots$}] () at (2.6,0.1) {};
\node [draw,circle,color=black, fill=black,inner sep=0pt,minimum size=5pt, label = \footnotesize$\ell$] (D) at (7,0) {};
\node [draw,circle,color=black, fill=white,,inner sep=0pt,minimum size=5pt,label = left:\footnotesize$B_1$] (DNW) at (6.5,0.5) {};
\node [draw,circle,color=black, fill=white,,inner sep=0pt,minimum size=5pt,label =left:\footnotesize$B_k$] (DSW) at (6.5,-0.5) {};
\node [label = {[label distance=-0cm]left:\footnotesize$\vdots$}] () at (6.4,0.1) {};
\node [draw,circle,color=black, fill=white,,inner sep=0pt,minimum size=5pt,label = right:\footnotesize$C_1$] (DSE) at (7.5,0.5) {};
\node [draw,circle,color=black, fill=white,,inner sep=0pt,minimum size=5pt,label = right:\footnotesize$C_l$] (DNE) at (7.5,-0.5) {};
\node [label = {[label distance=-0cm]right:\footnotesize$\vdots$}] () at (7.6,0.1) {};
\draw (A) -- (B) -- (C);
\draw (A) -- (ANW);
\draw (A) -- (ASW);
\draw (C) -- (CNE);
\draw (C) -- (CSE);
\node () at (4.5,0) {\footnotesize$\longleftrightarrow$};
\draw (D) -- (DSW);
\draw (D) -- (DNW);
\draw (D) -- (DSE);
\draw (D) -- (DNE);

\end{tikzpicture}
\caption{Removal/addition of a degree two white vertex in a double circuit configuration. Removal/addition of a degree two black vertex is similar.}\label{shrink4}
\end{figure}
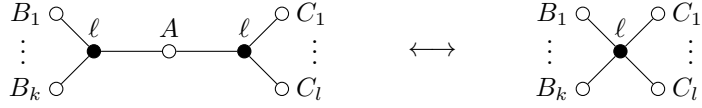
 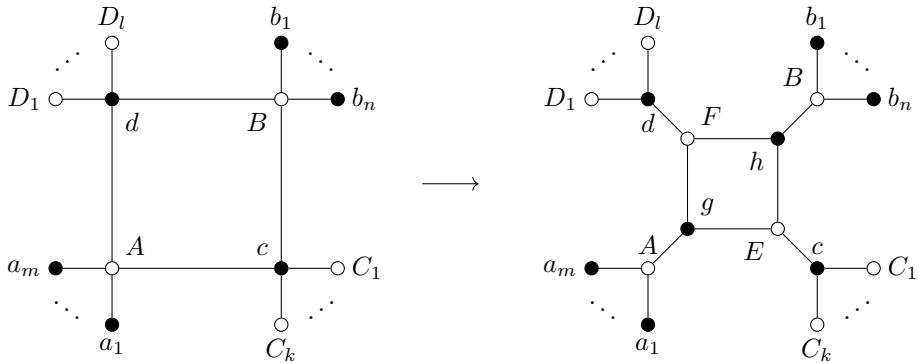
\begin{figure}[b]
 \centering
\begin{tikzpicture}[scale = 0.75, label distance=-0.5mm]
\node [draw,circle,color=black, fill=white,inner sep=0pt,minimum size=5pt, label = 45:\footnotesize$A$] (A) at (-9.5,0) {};
\node [draw,circle,color=black, fill=black,inner sep=0pt,minimum size=5pt, label = below:\footnotesize$a_1$] (A1) at (-9.5,-1) {};
\node [draw,circle,color=black, fill=black,inner sep=0pt,minimum size=5pt, label = left:\footnotesize$a_m$] (A2) at (-10.5,0) {};
\draw (A) -- (A1);
\draw (A) --(A2);
\node  () at (-10.3,-0.6) {\footnotesize$\ddots$};
\node [draw,circle,color=black, fill=black,inner sep=0pt,minimum size=5pt, label = 135:\footnotesize$c$] (B) at (-6.5,0) {};
\node [draw,circle,color=black, fill=white,inner sep=0pt,minimum size=5pt, label = right:\footnotesize$C_1$] (B1) at (-5.5,0) {};
\node [draw,circle,color=black, fill=white,inner sep=0pt,minimum size=5pt, label = below:\footnotesize$C_k$] (B2) at (-6.5,-1) {};
\node [] () at (-5.8,-0.6) {\footnotesize$\iddots$};
\draw (B) -- (B1);
\draw (B) -- (B2);
\node [draw,circle,color=black, fill=white,inner sep=0pt,minimum size=5pt, label = 225:\footnotesize$B$] (C) at (-6.5,3) {};
\node [draw,circle,color=black, fill=black,inner sep=0pt,minimum size=5pt, label = right:\footnotesize$b_n$] (C1) at (-5.5,3) {};
\node [draw,circle,color=black, fill=black,inner sep=0pt,minimum size=5pt, label = above:\footnotesize$b_1$] (C2) at (-6.5,4) {};
\draw (C) -- (C1);
\draw (C) -- (C2);
\node [] () at (-5.8,3.8) {\footnotesize$\ddots$};
\node [draw,circle,color=black, fill=black,inner sep=0pt,minimum size=5pt, label = -45:\footnotesize$d$] (D) at (-9.5,3) {};
\node [draw,circle,color=black, fill=white,inner sep=0pt,minimum size=5pt, label = left:\footnotesize$D_1$] (D1) at (-10.5,3) {};
\node [draw,circle,color=black, fill=white,inner sep=0pt,minimum size=5pt, label = above:\footnotesize$D_l$] (D2) at (-9.5,4) {};
\node [] () at (-10.3,3.8) {\footnotesize$\iddots$};
\draw (D) -- (D1);
\draw (D) -- (D2);
\draw [thin] (A) -- (B) node[midway, below] {$$}; \draw (B) -- (C) node[midway, right] {$$}; \draw (C) -- (D) node[midway, above] {$$}; \draw (D) -- (A) node[midway, left] {$$};;
\node [draw,circle,color=black, fill=white,inner sep=0pt,minimum size=5pt, label = 90:\footnotesize$A$] (A) at (0,0) {};
\node [draw,circle,color=black, fill=black,inner sep=0pt,minimum size=5pt, label = \footnotesize$c$] (B) at (3,0) {};
\node [draw,circle,color=black, fill=white,inner sep=0pt,minimum size=5pt, label = 135:\footnotesize$B$] (C) at (3,3) {};
\node [draw,circle,color=black, fill=black,inner sep=0pt,minimum size=5pt, label = below:\footnotesize$d$] (D) at (0,3) {};
\node [draw,circle,color=black, fill=black,inner sep=0pt,minimum size=5pt, label = 45:\footnotesize$g$] (A1) at (0.7,0.7) {};
\node [draw,circle,color=black, fill=white,inner sep=0pt,minimum size=5pt, label =-135: \footnotesize$E$] (B1) at (2.3,0.7) {};
\node [draw,circle,color=black, fill=black,inner sep=0pt,minimum size=5pt, label = 225:\footnotesize$h$] (C1) at (2.3,2.3) {};
\node [draw,circle,color=black, fill=white,inner sep=0pt,minimum size=5pt, label =45:\footnotesize$F$] (D1) at (0.7,2.3) {};
\draw [thin] (A1) -- (B1) node[midway, below] {$$}; \draw (B1) -- (C1) node[midway, right] {$$}; \draw (C1) -- (D1) node[midway, above] {$$}; \draw (D1) -- (A1) node[midway, left] {$$};;
\draw (A) -- (A1);
\draw (B) -- (B1);
\draw (C) -- (C1);
\draw (D) -- (D1);
\node [draw,circle,color=black, fill=black,inner sep=0pt,minimum size=5pt, label = below:\footnotesize$a_1$] (A1) at (0,-1) {};
\node [draw,circle,color=black, fill=black,inner sep=0pt,minimum size=5pt, label = left:\footnotesize$a_m$] (A2) at (-1,0) {};
\draw (A) -- (A1);
\draw (A) --(A2);
\node  () at (-0.8,-0.6) {\footnotesize$\ddots$};
\node [draw,circle,color=black, fill=white,inner sep=0pt,minimum size=5pt, label = right:\footnotesize$C_1$] (B1) at (4,0) {};
\node [draw,circle,color=black, fill=white,inner sep=0pt,minimum size=5pt, label = below:\footnotesize$C_k$] (B2) at (3,-1) {};
\node [] () at (3.7,-0.6) {\footnotesize$\iddots$};
\draw (B) -- (B1);
\draw (B) -- (B2);
\node [draw,circle,color=black, fill=black,inner sep=0pt,minimum size=5pt, label = right:\footnotesize$b_n$] (C1) at (4,3) {};
\node [draw,circle,color=black, fill=black,inner sep=0pt,minimum size=5pt, label = above:\footnotesize$b_1$] (C2) at (3,4) {};
\draw (C) -- (C1);
\draw (C) -- (C2);
\node [] () at (3.7,3.8) {\footnotesize$\ddots$};
\node [draw,circle,color=black, fill=white,inner sep=0pt,minimum size=5pt, label = left:\footnotesize$D_1$] (D1) at (-1,3) {};
\node [draw,circle,color=black, fill=white,inner sep=0pt,minimum size=5pt, label = above:\footnotesize$D_l$] (D2) at (0,4) {};
\node [] () at (-0.8,3.8) {\footnotesize$\iddots$};
\draw (D) -- (D1);
\draw (D) -- (D2);
\draw [-> ](-4,1.5) -- (-3, 1.5);

\end{tikzpicture}
\caption{An urban renewal of a double circuit configuration. Here $E= \langle A, B\rangle \cap \langle C_1, \dots, C_k \rangle$, $F= \langle A, B\rangle \cap \langle D_1, \dots, D_l \rangle$, $g = \langle c \cap d, a_1 \cap \dots \cap a_m\rangle$, $h= \langle c \cap d, b_1 \cap \dots \cap b_n\rangle$.}\label{urban3}
\end{figure}

Our main objects of study are \emph{coherent double circuit configurations}.

\begin{theorem}\label{mainThm}
Moves preserve coherence of double circuit configurations.
\end{theorem}
\begin{proof}
For addition/removal of a degree two vertex, this follows from the relation
$$
[\dots, a, B, a, \dots ] = [\dots, a, \dots]
$$
for the multi-ratio (along with an analogous relation with points and hyperplanes interchanged). 
 \begin{figure}[t]
 \centering
\begin{tikzpicture}[scale = 0.75, , label distance=-0.5mm]
\node [draw,circle,color=black, fill=white,inner sep=0pt,minimum size=5pt, label = 135:\footnotesize$A$] (A) at (0,0) {};
\node [draw,circle,color=black, fill=black,inner sep=0pt,minimum size=5pt, label = 45:\footnotesize$c$] (B) at (3,0) {};
\node [draw,circle,color=black, fill=white,inner sep=0pt,minimum size=5pt, label = 135:\footnotesize$B$] (C) at (3,3) {};
\node [draw,circle,color=black, fill=black,inner sep=0pt,minimum size=5pt, label = 225:\footnotesize$d$] (D) at (0,3) {};
\node [draw,circle,color=black, fill=black,inner sep=0pt,minimum size=5pt, label = 45:\footnotesize$g$] (A1) at (0.7,0.7) {};
\node [draw,circle,color=black, fill=white,inner sep=0pt,minimum size=5pt, label =135: \footnotesize$E$] (B1) at (2.3,0.7) {};
\node [draw,circle,color=black, fill=black,inner sep=0pt,minimum size=5pt, label = 225:\footnotesize$h$] (C1) at (2.3,2.3) {};
\node [draw,circle,color=black, fill=white,inner sep=0pt,minimum size=5pt, label =-45:\footnotesize$F$] (D1) at (0.7,2.3) {};
\draw [ color = black] (A) -- (B) -- (C) -- (D)  -- (A);
\draw [thin] (A1) -- (B1) node[midway, below] {$$}; \draw (B1) -- (C1) node[midway, right] {$$}; \draw (C1) -- (D1) node[midway, above] {$$}; \draw (D1) -- (A1) node[midway, left] {$$};;
\draw (A) -- (A1);
\draw (B) -- (B1);
\draw (C) -- (C1);
\draw (D) -- (D1);
\node [draw,circle,color=black, fill=black,inner sep=0pt,minimum size=5pt, label = below:\footnotesize$a_1$] (A1) at (0,-1) {};
\node [draw,circle,color=black, fill=black,inner sep=0pt,minimum size=5pt, label = left:\footnotesize$a_m$] (A2) at (-1,0) {};
\draw (A) -- (A1);
\draw (A) --(A2);
\node  () at (-0.8,-0.6) {\footnotesize$\ddots$};
\node [draw,circle,color=black, fill=white,inner sep=0pt,minimum size=5pt, label = right:\footnotesize$C_1$] (B1) at (4,0) {};
\node [draw,circle,color=black, fill=white,inner sep=0pt,minimum size=5pt, label = below:\footnotesize$C_k$] (B2) at (3,-1) {};
\node [] () at (3.7,-0.6) {\footnotesize$\iddots$};
\draw (B) -- (B1);
\draw (B) -- (B2);
\node [draw,circle,color=black, fill=black,inner sep=0pt,minimum size=5pt, label = right:\footnotesize$b_n$] (C1) at (4,3) {};
\node [draw,circle,color=black, fill=black,inner sep=0pt,minimum size=5pt, label = above:\footnotesize$b_1$] (C2) at (3,4) {};
\draw (C) -- (C1);
\draw (C) -- (C2);
\node [] () at (3.7,3.8) {\footnotesize$\ddots$};
\node [draw,circle,color=black, fill=white,inner sep=0pt,minimum size=5pt, label = left:\footnotesize$D_1$] (D1) at (-1,3) {};
\node [draw,circle,color=black, fill=white,inner sep=0pt,minimum size=5pt, label = above:\footnotesize$D_l$] (D2) at (0,4) {};
\node [] () at (-0.8,3.8) {\footnotesize$\iddots$};
\draw (D) -- (D1);
\draw (D) -- (D2);

\end{tikzpicture}
\caption{To the proof of Theorem \ref{mainThm}.}\label{urban4}
\end{figure}
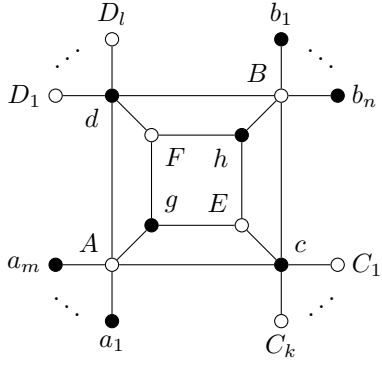
To prove that urban renewals preserve coherence, consider Figure \ref{urban3} and assume that the tiling on the left is coherent. Denote that tiling by $\mathcal L$. We aim to prove that the tiling on the right, which we'll denote by $\mathcal R$, is also coherent. To that end, consider an intermediate tiling (which is not a double circuit configuration!) $\mathcal I$ shown in Figure \ref{urban4}. 
%
Note that $\mathcal I$ has the same tiles as $\mathcal L$ except for the five tiles in the middle. So, to prove coherence of $\mathcal I$, it suffices to prove coherence of those five tiles. Furthermore, by Theorem \ref{mthm}, it is sufficient to establish coherence of four out of these five tiles. Consider, for instance, the tile $AdFg$. By definition of the point $F$, it belongs to the line $AB$, so $AF = AB$. Likewise $g \cap d = c \cap d$. So, $AF \cap g \cap d = AB\cap c \cap d$. The latter intersection is non-trivial since $AcBd$ is a tile in the tiling $\mathcal L$. So, the former intersection is also non-trivial, and the tile $AdFg$ is coherent. In the same way we get coherence of tiles $BdFh, BhEc, AgEc$. So, tiling $\mathcal I$ is indeed coherent. But $\mathcal R$ is obtained from $\mathcal I$ by gluing adjacent tiles, and so is coherent as well.
\end{proof}

\section{Example: pentagram map}\label{sec:ex1}

Consider a  more general version of the pentagram map where instead of short diagonals we use diagonals $P_i P_{i+k}$ for some fixed $k\in \{2, \dots, n-2\}$. Specifically, the evolution of an $n$-gon $P$ under such map is given by 
$T_k(P) = P'$, where 
$$P_i' = P_i P_{i+k} \cap P_{i+1} P_{i+k+1}.$$
In this notation the pentagram map defined above is just the special case $T = T_2$. 

The map $T_2$ was introduced by Schwartz \cite{Sch}. There is a growing body of literature on this subject, and we do not attempt to provide all the references. The maps $T_k$ for $k > 2$ appear, e.g., in \cite{gekhtman2011, izosimov2023long, glick2016}.
A variant of the map $T_k$ where one intersects non-consecutive diagonals is also studied in  \cite{zou2024}.



\begin{figure}[htbp]
    \centering
    \includegraphics[width=0.75\textwidth]{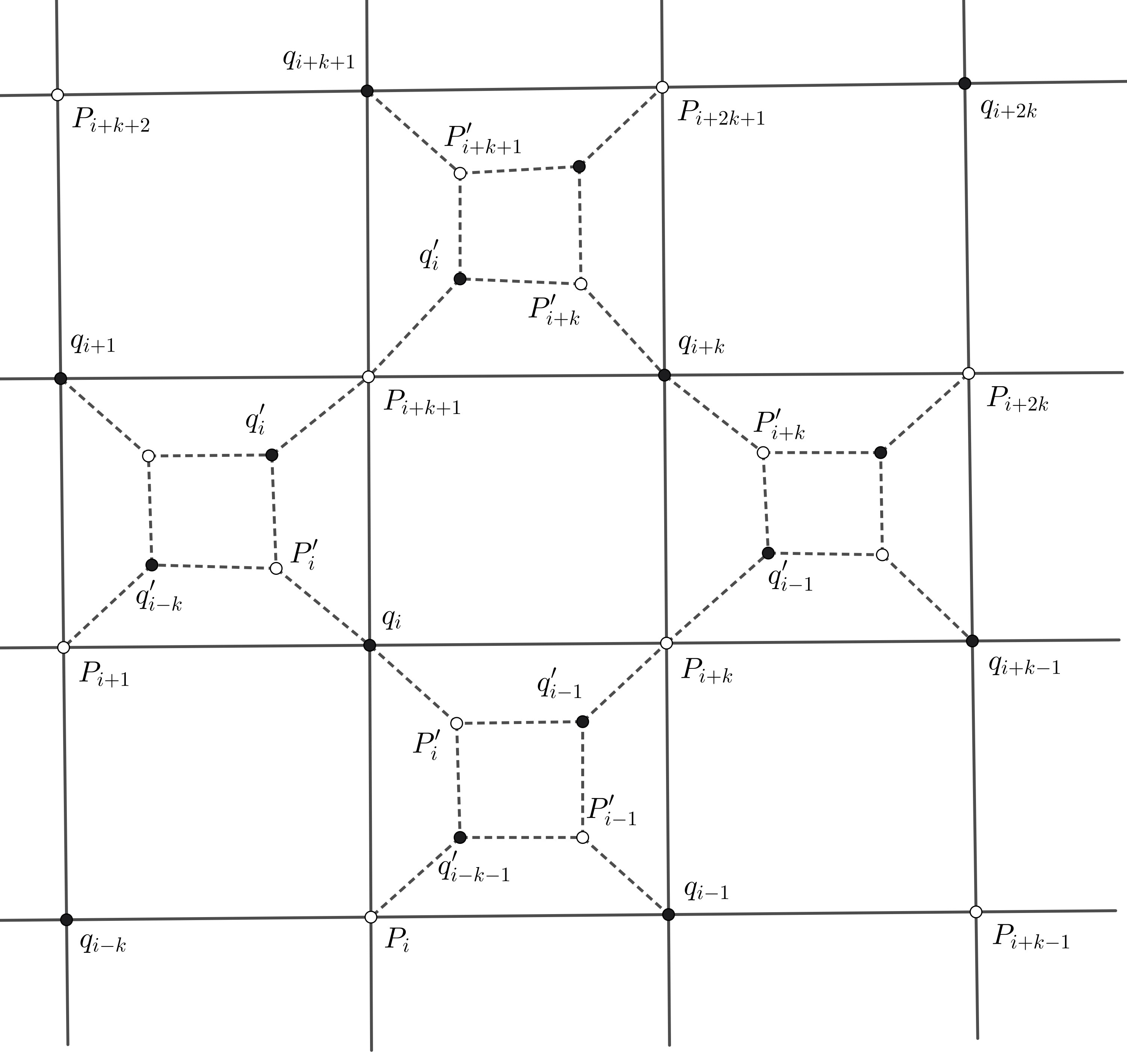} 
    \caption{The bipartite graph corresponding to the pentagram map $T_k$.}
    \label{fig:pentagram}
\end{figure}

One can also consider (dual) polygons defined by $n$-tuples of lines (as opposed to points), and the following pentagram dynamics $T_k(Q) = Q'$ on them: $$q_i' = \langle q_i \cap q_{i+1}, q_{i+k} \cap q_{i+k+1} \rangle.$$  Note that, in terms of line indices, the time direction is reversed: the forward time direction on $Q$ matches the reverse time direction on $P$. This changes however if we were to keep the data for $Q$ in terms of points rather than lines. Indeed,
define vertices of $Q$ as $Q_i = q_i \cap q_{i-k}$. Then $q_i \cap q_{i+1}$ lies on both $q_i'$ and $q_{i-k}'$, thus $Q_i'  = q_i' \cap q_{i-k}' = q_i \cap q_{i+1} = Q_i Q_{i+k} \cap Q_{i+1} Q_{i+k+1}$, which is the usual $T_k$ dynamics.

Consider now the bipartite graph on a torus shown in Figure \ref{fig:pentagram} (for now, ignore the dashed edges and vertices with primed labels). We plant vertices of $P$ and sides of $Q$ at white and black vertices respectively, and we index so that the indices grow by $k$ in the north-east direction and by $1$ in the north-west  direction. Then we quotient by the appropriate sublattice to identify vertices with the same index. The condition (V) is satisfied automatically, as the whole construction lives in the plane, and thus any four generic points/lines form a circuit. 

\begin{lemma}
Condition \emph{(F)} for this graph is equivalent to two conditions holding at the same time: $Q$ being inscribed into $P$ and $T_k(Q)$ being inscribed into $T_k(P)$. 
\end{lemma}

\begin{proof}
 We apply  Proposition \ref{prop:coh} in the planar case to the two types of tiles/faces we have in our graph: quadrilaterals  $q_{i-k} P_{i+1} q_i P_i$ and quadrilaterals $P_i q_i P_{i+k} q_{i-1}$. Coherence of the former means that line $P_i P_{i+1}$ passes through the point of intersection of $q_i$ with $q_{i-k}$, which is equivalent to $Q_i \in P_i P_{i+1}$. Coherence of the latter kind of tile means that line $P_i P_{i+k}$ passes through the point of intersection of $q_i$ with $q_{i-1}$. As $P_i', P_{i-1}' \in P_i P_{i+k}$ and $q_i \cap q_{i-1} \in q_{i-1}', q_{i-k-1}'$, this means that $Q_{i-1}' \in P_{i-1}' P_i'$. Thus $Q$ is inscribed into $P$ and $Q' = T_k(Q)$ is inscribed into $P' = T_k(P)$. 
\end{proof}

Now we consider dynamics by performing urban renewal at half of the tiles, followed by removal of degree $2$ vertices, as shown in Figure \ref{fig:pentagram}. 

\begin{proposition} \label{prop:pent}
 This dynamics coinsides with pentagram dynamics $T_k$ on $P$ and $Q$. 
\end{proposition}

\begin{proof}
 One directly checks it from the rules of urban renewal. For example, according to the urban renewal rules one of the new white vertices inside the tile $P_i q_i P_{i+k} q_{i-1}$ should indeed be the intersection of $P_i P_{i+k}$ with $P_{i+1}P_{i+k+1}$, i.e. $P_i'$.
\end{proof}

For a single pentagram dynamics on $P$ and for $k=2$ Proposition \ref{prop:pent} was observed in \cite[Example 4.4]{affolter2024vector}.

Below is our main dynamic theorem for pentagram maps, generalizing Theorem \ref{thm1}.

\begin{theorem}\label{thm:pent}
Let $T_k$ be the $k$-diagonal pentagram map, and $P$, $Q$ be planar polygons such that $Q$ is inscribed in $P$, and $T_k(Q)$ in $T_k(P)$. Then $T_k^2(Q)$ is also inscribed in $T_k^2(P)$, and thus $T_k^m(Q)$ is inscribed in $T_k^m(P)$ for any $m \geq 1$. 
\end{theorem}

\begin{proof}
 Follows from Theorem \ref{mainThm} and Proposition \ref{prop:pent}. 
\end{proof}

 \begin{figure}[t]
\centering
  \centering
\includegraphics[scale = 0.25]{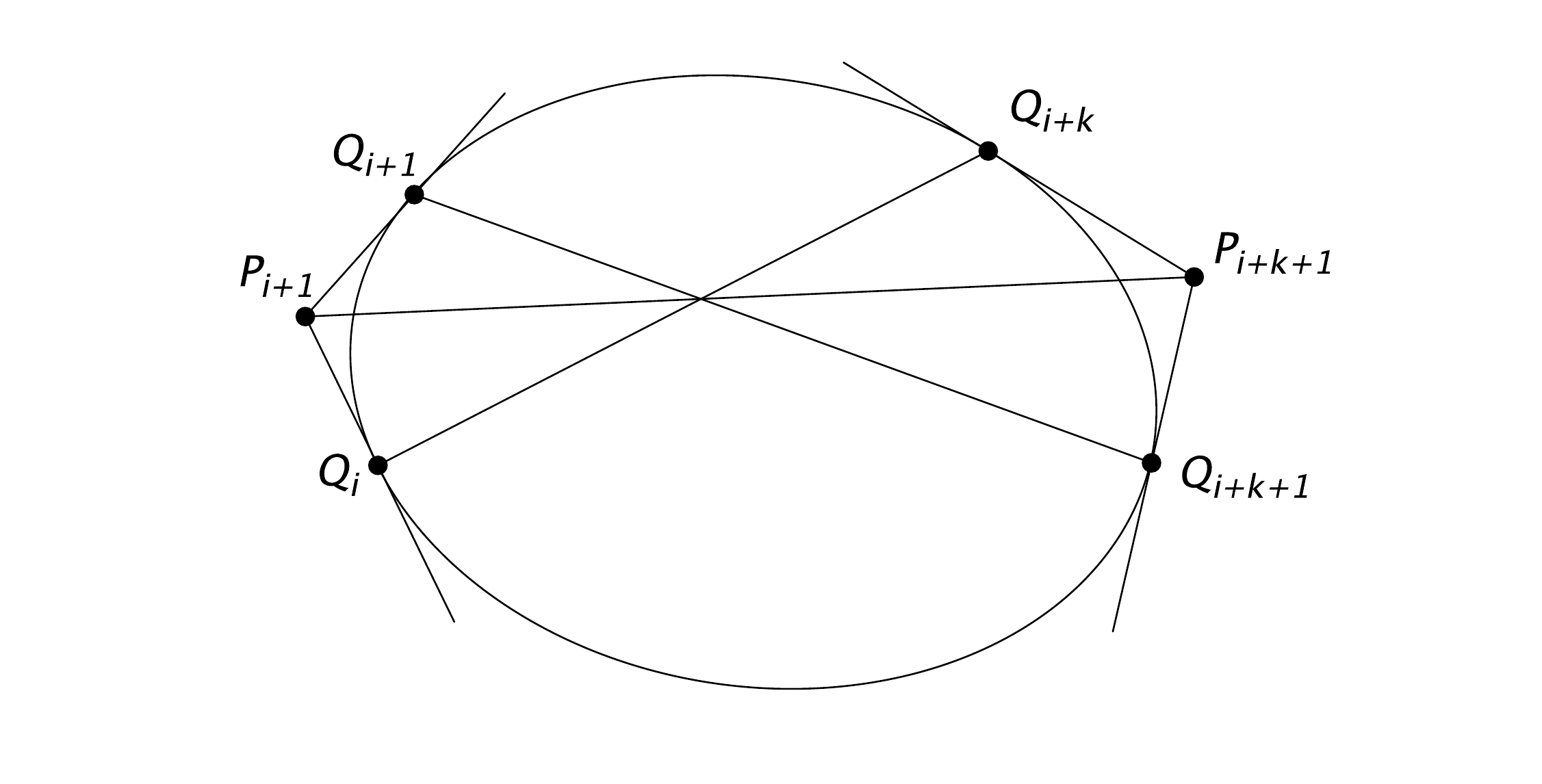}

\caption{If $P$ and $Q$ are dual, then $T_k(Q)$ is inscribed in $T_k(P)$.
}\label{fig:dual}
\end{figure}

The following result gives an example of polygons $P$, $Q$ which can serve as initial data for Theorem \ref{thm:pent}. Say that polygons $P$ and $Q$ are \emph{dual} to each other if $P$ is circumscribed about a conic $C$ and $Q_i$ is the point where $P_iP_{i+1}$ touches $C$ (equivalently, $P$ is circumscribed about a conic $C$, and $Q$ is polar dual to $P$ with respect to $C$). Clearly, if $P$ and $Q$ are dual, then $Q$ is inscribed in $P$.
\begin{proposition}\label{thm:ins}
Suppose $P$ and $Q$ are dual polygons. Then $T_k(Q)$ is inscribed in $T_k(P)$ for all $k\in \{2, \dots, n-2\}$, see Figure \ref{fig:dual}.\end{proposition}
The  proof is based on the following folklore statement:
  \begin{figure}[t]
\centering
  \centering
\includegraphics[scale = 0.3]{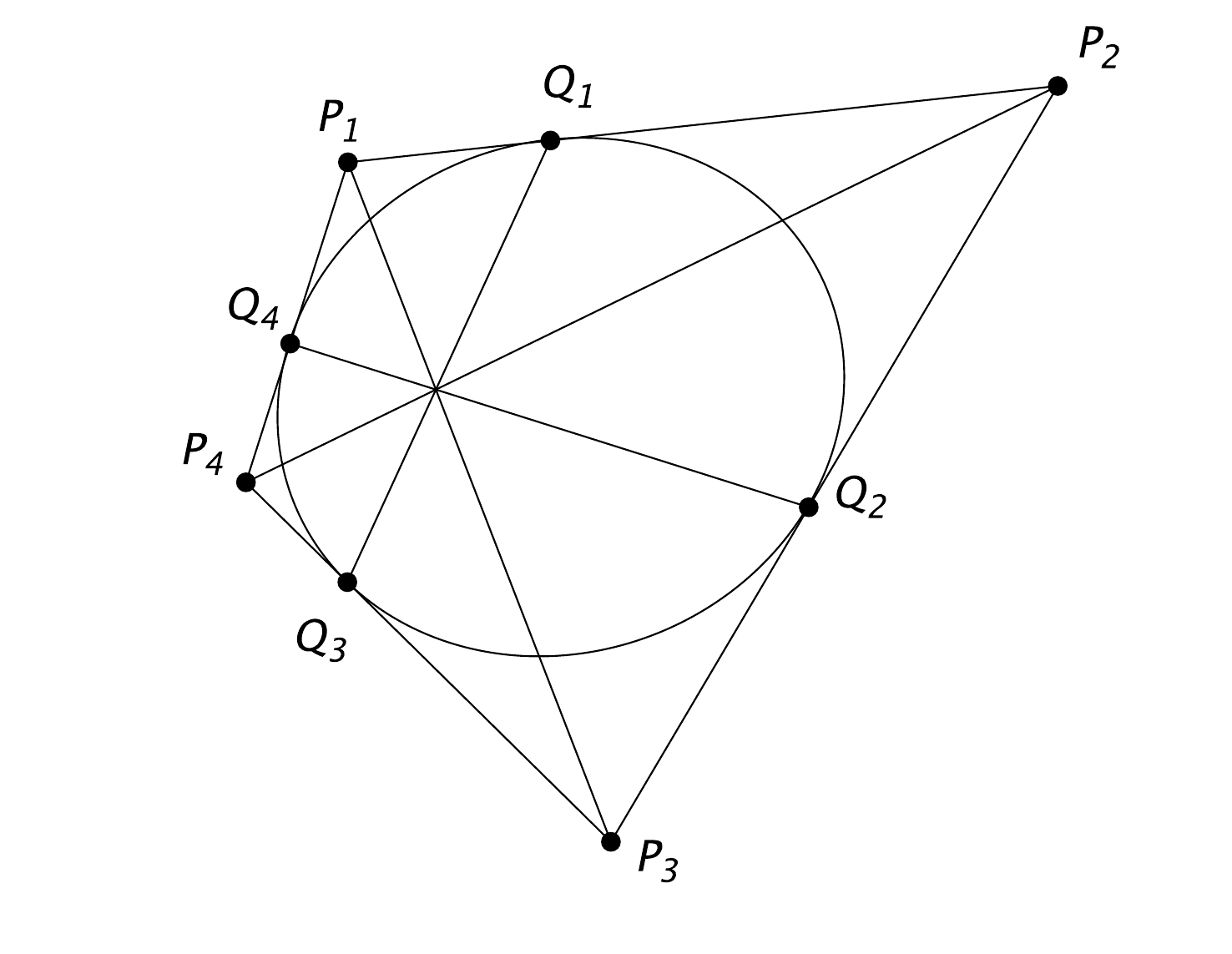}

\caption{Dual quadrilaterals theorem.
}\label{fig:quaddiag}
\end{figure}
\begin{theorem}[Dual quadrilaterals theorem, see e.g. \cite{bobenko2008}, Corollary 9.20]
Suppose $P$ and $Q$ are dual quadrilaterals. Then the diagonals of $P$ and $Q$ are concurrent, see Figure \ref{fig:quaddiag}.
\end{theorem}
\begin{proof}
Suppose that the side $P_iP_{i+1}$ touches a conic $C$ at $Q_i$ (where indices are understood modulo $4$). Consider a degenerate hexagon with vertices $P_1, Q_1, P_2, P_3, Q_3, P_4$. All its sides are tangent to $C$, so, by Brianchon's theorem, its diagonals $P_1P_3, Q_1Q_3, P_2P_4$ are  concurrent. Analogously, $P_1P_3, Q_2Q_4, P_2P_4$  are also concurrent. The result follows. 
\end{proof}
\begin{proof}[{Proof of Theorem \ref{thm:ins}}]
Suppose that the side $P_iP_{i+1}$ touches a conic $C$ at $Q_i$. Consider a quadrilateral with sides $P_{i}P_{i+1}, P_{i+1}P_{i+2}, P_{i+k}P_{i+k+1}, P_{i+k+1}P_{i+k+2}$ which are tangent to $C$ at points $Q_i, Q_{i+1}, Q_{i+k}, Q_{i+k+1}$ respectively. Notice that $P_{i+1}, P_{i+k+1}$ are two opposite vertices of this quadrilateral. So, by the dual quadrilaterals theorem, the lines $Q_{i}Q_{i+k}, Q_{i+1}Q_{i+k+1}, P_{i+1}P_{i+k+1}$ are concurrent, as desired.  \end{proof}

\section{Example: spirals}\label{sec:spirals}

In \cite{schwartz2013} Schwartz defined the following object, see also \cite{beffa2015integrability}. Let $P$ be a bi-infinite polygonal path in ${\mathbb R}P^2$. We say that $P$ is a pentagram spiral if for some $m$ we have $T^m(P) = P$, where $T$ is the usual pentagram map, which can be applied to $P$ in a natural way. The minimal such number $m$ is called in \cite{schwartz2013} the {\it {order}} of $P$. 
Here we restrict our attention to the case of order $m=1$, however allow higher pentagram maps $T_k$ as part of the definition. Thus, for us a {\it {pentagram spiral}} is a bi-infinite polygonal path $P$ such that $T_k(P) = P$. Equivalently, it is a bi-infinite polygonal path $P$ such that $$P_{i+n} = P_i P_{i+k} \cap P_{i+n-1} P_{i+k-1}.$$ Denote $\pp_{k,n}$ the set of all such $P$. For $k=2$ those are Schwartz's pentagram spirals of type $(n-1, 1)$. The set $\pp_{k,n}$ is defined for any $k \geq 2$ and $n > k+2$ (for $n=k+1, k+2$ one can check that the sequence $P_i$ becomes periodic). 

\begin{remark}
As we shall soon see, $P_{i+n} = P_i P_{i+k} \cap P_{i+n-1} P_{i+k-1}$ is equivalent to $P_{i+n} = P_i P_{i+k} \cap P_{i-1} P_{i+k-1}$, which up to a slight adjustment is how the pentagram map is written in the previous section. This is because points $P_{i-1}, P_{i+n-1}, P_{i+k-1}$ are collinear.  
\end{remark}

One can also consider dual bi-infinite polygonal paths defined by lines (as opposed to points), and the following defining recurrence on them: $$q_{i+n} = \langle q_i \cap q_{i+n-k}, q_{i+k} \cap q_{i+k-1} \rangle.$$ As before, we can also keep the data for $Q$ in terms of points rather than lines. Indeed,
define vertices of $Q$ as $Q_i = q_i \cap q_{i-k}$. 

\begin{figure}[htbp]
    \centering
    \includegraphics[width=\textwidth]{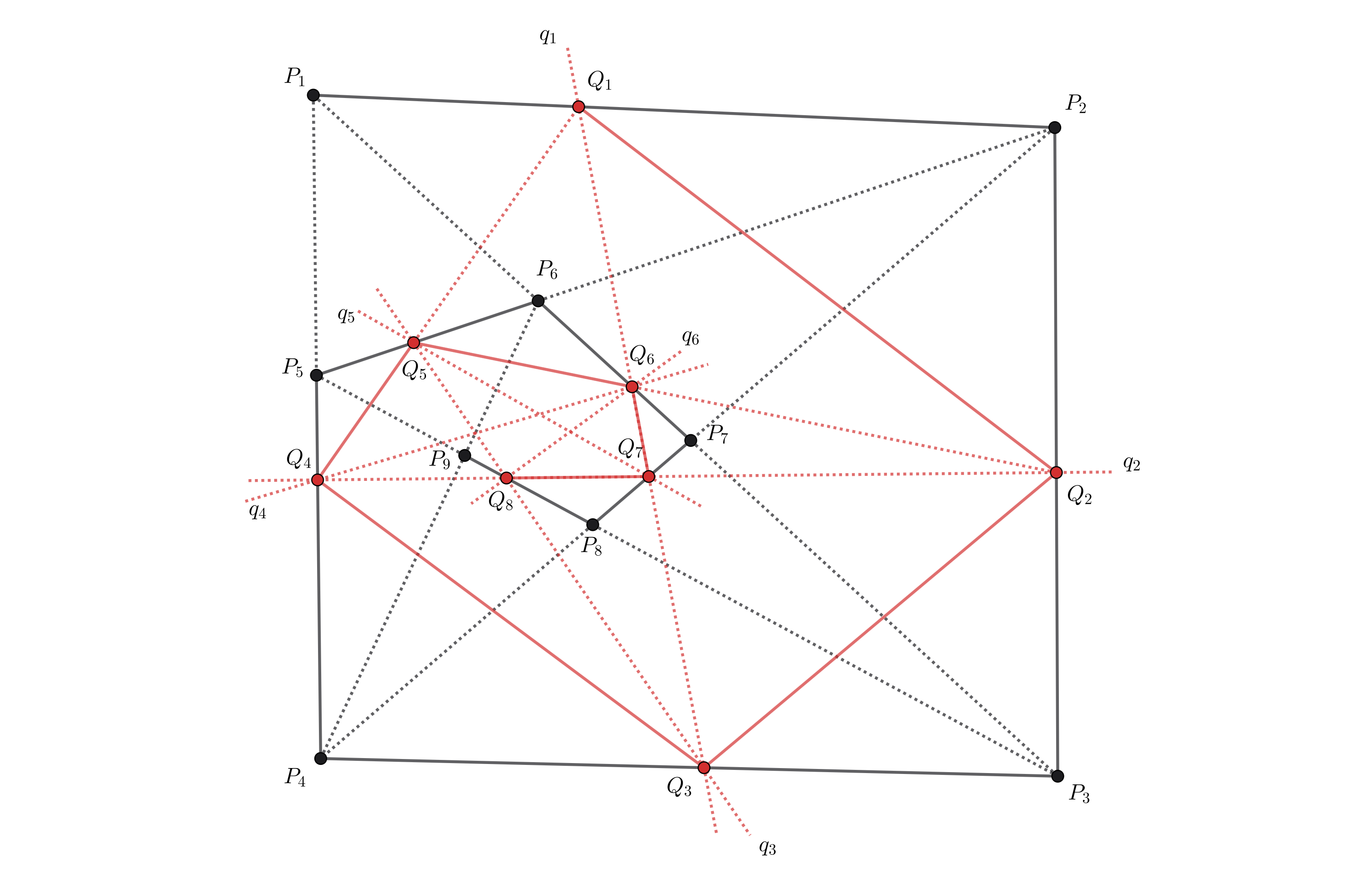} 
    \caption{A pair of pentagram spirals $P, Q \in \pp_{2,5}$, $Q$ is inscribed into $P$.}
    \label{fig:spiral}
\end{figure}

By definition, $Q_{i+n} = q_{i+n} \cap q_{i+n-k}$, and since $q_{i+n}$ (also by definition) passes through intersection of $q_i$ with $q_{i+n-k}$, we conclude that $q_i$ passes through $Q_{i+n}$. It is easy to see that $q_i=Q_{i} Q_{i+k}$. 
Relabelling the index this also tells us that $Q_{i+n-1} \in q_{i-1}$. On the other hand by definition of $Q_{i+k-1}$ we know $Q_{i+k-1} \in q_{i-1}$. Thus $Q_{i+n-1}Q_{i+k-1} = q_{i-1}$. We know that $q_{i} \cap q_{i-1} \in q_{i+n-k}$ by the definition of the latter, which means that $q_{i-1}$ passes through $q_{i} \cap q_{i+n-k}$, which we already know to be $Q_{i+n}$. We see that $Q_{i+n}$ lies on both $Q_{i} Q_{i+k}$ and $Q_{i+n-1}Q_{i+k-1}$, and thus $$Q_{i+n} = Q_i Q_{i+k} \cap Q_{i+n-1} Q_{i+k-1}$$ as desired. 

Note that as a byproduct of the above argument we conclude that $Q_{i+n} = q_i \cap q_{i-1}$. 

\begin{example}
Figure \ref{fig:spiral} shows two spirals $P,Q \in \pp_{2,5}$ where $P$ is given by its points, while $Q$ is given both by its lines and points. 
\end{example}

Instead of thinking of $P \in \pp_{k,n}$ as a bi-infinite polygonal path, we can generate it from a finite initial seed as follows. A {\it {seed}} of an element of $\pp_{k,n}$ is a window of $n+1$ points $[P_i, \ldots, P_{i+n}]$ subject to the following conditions: for any $0 \leq \ell \leq k-1$ points $P_{i+\ell}$, $P_{i+n-k+1+\ell}$, $P_{i+n-k+\ell}$ are collinear, and also $P_i, P_{i+k}, P_{i+n}$ are collinear. 

\begin{lemma}
 Given a seed of a pentagram spiral, one can compute further and previous points in it via recursions 
 $$P_{i+n+1} = P_{i+1} P_{i+k+1} \cap P_{i+n} P_{i+k} \text{   and   } P_{i-1} = P_{i+n-k-1} P_{i+n-k} \cap P_{i+n-1} P_{i+k-1}.$$
 After reindexing $i \mapsto i+1$ or $i \mapsto i-1$ one gets a new seed. 
\end{lemma}

\begin{proof}
 If we are moving forward, it is clear that the recursion gives us the desired $P_{i+n+1}$ by the definition of pentagram spirals. The new shifted window indeed forms a seed since the only new conditions to check are collinearity of $P_{i+k}$, $P_{i+n+1}$, $P_{i+n}$, which is true directly from the recursion, and collinearity of  $P_{i+1}, P_{i+k+1}, P_{i+n+1}$, which is also true directly from the recursion.

 If we are moving backwards, from $P_{i+n-k} = P_{i-k} P_{i} \cap P_{i+n-1-k} P_{i-1}$ we see that that $P_{i-1}, P_{i+n-1-k}, P_{i+n-k}$ are collinear, and from $P_{i+n-1} = P_{i-1} P_{i+k-1} \cap P_{i+n-2} P_{i+k-2}$ we see that $P_{i-1}, P_{i+k-1}, P_{i+n-1}$ are collinear. Thus the backwards recursion indeed produces $P_{i-1}$. To know that we got a new seed we need to check that $P_{i-1}, P_{i+n-1-k}, P_{i+n-k}$ are collinear, which we already know, and also that $P_{i-1}, P_{i+k-1}, P_{i+n-1}$ are collinear, which we also already know. 
\end{proof}

\begin{remark}
In fact, to recover the whole spiral it is sufficient to start with a shorter initial seed, $[P_i, \ldots, P_{i+n-1}]$. We define slightly longer seeds to align better with what follows. 
\end{remark}

\begin{example}
Spiral $P$ in Figure \ref{fig:spiral}  is determined by the initial seed $[P_1, \ldots, P_6]$ satisfying additional $k+1=3$ collinearity conditions: $P_1, P_4, P_5$ are collinear, $P_2, P_5, P_6$ are collinear, $P_1, P_3, P_6$ are collinear. All vertices of the spiral can be recovered from this seed, for example $P_7 = P_3P_6 \cap P_2P_4$, etc. 
\end{example}

Similarly, in terms of lines a pentagram spiral is determined by a {\it {seed}} which is a window $[q_i, \ldots, q_{i+n}]$ subject to the condition: for any $0 \leq \ell \leq k-1$ lines $q_{i+\ell}$, $q_{i+\ell+1}$, $q_{i+n-k+\ell+1}$ are collinear, and also $q_{i}$, $q_{i+n-k}$, $q_{i+n}$ are collinear. The recursions allowing shifting the seed index and thus computing the whole spiral out of the seed are 
$$q_{i+n+1} = \langle q_{i+1} \cap q_{i+n-k+1}, q_{i+k+1} \cap q_{i+k} \rangle \text{   and   } q_{i-1} = \langle q_i \cap q_{i+n-k}, q_{i+n-1} \cap q_{i+n-k-1} \rangle.$$
The proof is similar and is omitted for brevity. 

\begin{figure}[t]
    \centering
    \includegraphics[width=0.65\textwidth]{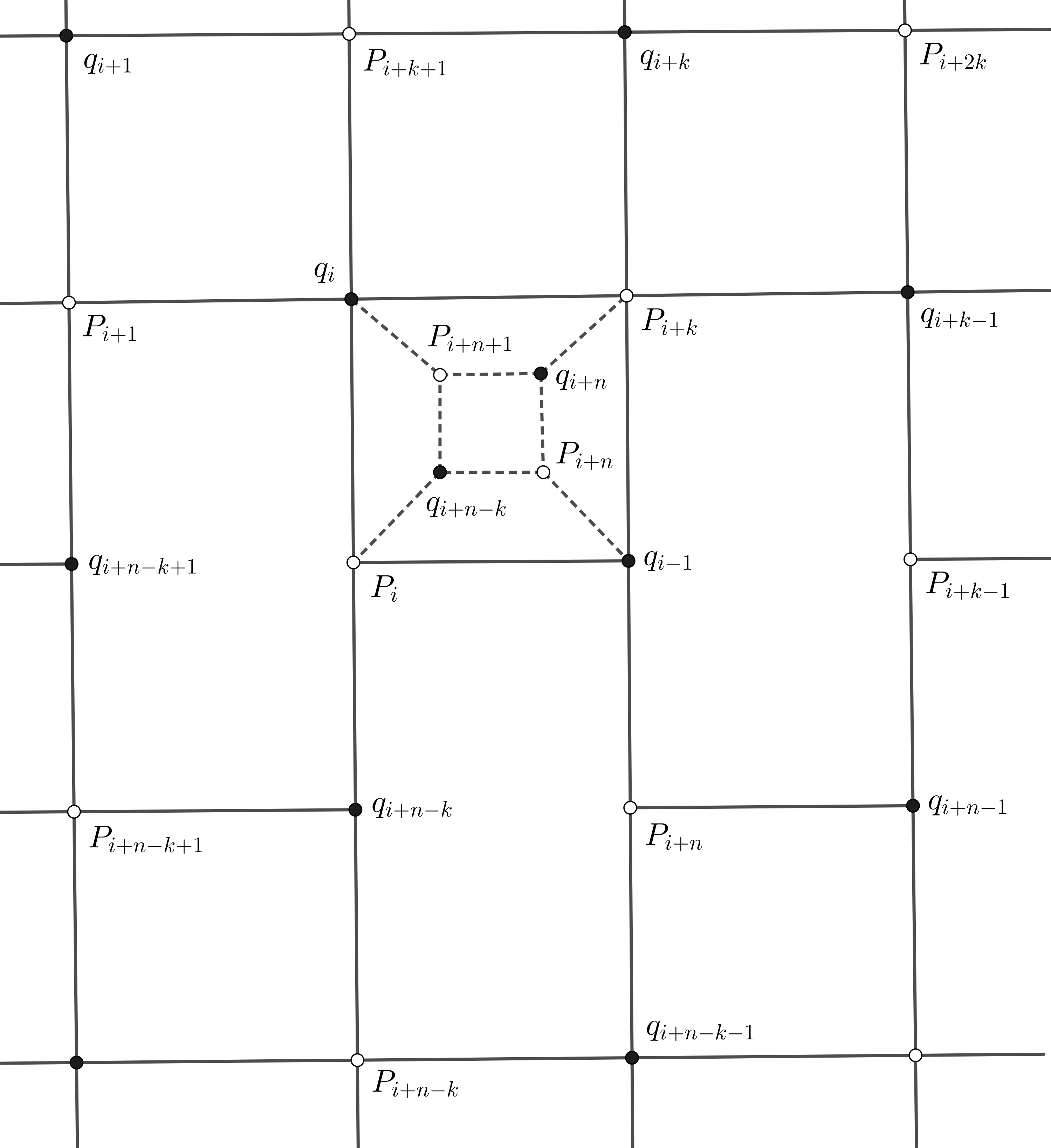} 
    \caption{A fragment of the toric bipartite graph for spirals in $\pp_{k,n}$.}
    \label{fig:spiralgraph}
\end{figure}

Now consider two pentagram spirals: $P, Q \in \pp_{k,n}$ given by their seeds $[P_i, \ldots, P_{i+n}]$ and $[q_{i-1}, \ldots,  q_{i+n-1}]$ respectively. Consider a bipartite graph on a torus determined as follows:
\begin{itemize}
  \item we start with the same graph as in Figure \ref{fig:pentagram} labelled with elements of the two seeds so that indices grow by $k$ in the north-east direction and by $1$ in the north-west  direction, modulo $n+1$;
  \item we remove edge $q_j P_{j+k}$ from the graph for each $j = n+i-k, \ldots, n+i-1, i-1$. 
\end{itemize}
A local part of such a graph is shown in Figure \ref{fig:spiralgraph} (as in the previous section, dashed edges along with vertices labeled $P_{i+n}, P_{i+n+1}, q_{i+n}, q_{i+n-k}$ only arise when we consider dynamics and for now can be ignored).

\begin{proposition}
The conditions (V) and (F) for this bipartite graph are equivalent to the following collection of conditions:
\begin{enumerate}
  \item the seeds $[P_i, \ldots, P_{i+n}]$ and $[q_{i-1}, \ldots,  q_{i+n-1}]$ satisfy all the necessary conditions in the definition of a seed;
  \item point $Q_j$ of the second pentagram spiral lies on $P_j P_{j+1}$ for $i \leq j \leq i+2n-k$. 
\end{enumerate}
\end{proposition}

\begin{proof}
Each missing edge $q_j P_{j+k}$ creates two vertices of degree three. The condition (V) for each such vertex labelled by $q_j$ means that $P_{j+1}$, $P_j$, $P_{j+k+1}$ are collinear, where the indices are taken modulo $n+1$.  As $j = n+i-k, \ldots, n+i-1, i-1$, this is equivalent to collinearity of points $P_{i+\ell}$, $P_{i+n-k+1+\ell}$, $P_{i+n-k+\ell}$ for $0 \leq \ell \leq k-1$, and also of $P_i, P_{i+k}, P_{i+n}$. The case of vertex labelled by $P_{j+k}$ is similar. 

Now let us consider the types of faces the graph has and what condition (F) means for each of them. The square faces with vertices $P_j, q_j, P_{j+k}, q_{j-1}$ tell us that $Q_{j+n} = q_j \cap q_{j-1}$ lies on $P_j P_{j+k} = P_{j+n} P_{j+n+1}$. Thus we see that $Q_j$ lies on $P_j P_{j+1}$ for $i+n \leq j \leq i+2n-k-1$. Similarly square faces with vertices $q_j, P_{j+k+1}, q_{j+k}, P_{j+k}$ tell us that $Q_{j+k} = q_j \cap q_{j+k}$ lies on $P_{j+k} P_{j+k+1}$, which means the desired condition of $Q_j$ lying on $P_j P_{j+1}$ for $i+k \leq j \leq i+n-1$. Finally, hexagons with vertices $P_j, q_j, P_{j+k+1}, q_{j+k}, P_{j+k}, q_{j-1}$ can be interpreted using Proposition \ref{prop:des}. As $q_j \cap q_{j+k} = Q_{j+k}$, and we already know that $Q_{j+k} \in P_{j+k} P_{j+k+1}$, the Desargues condition implies that $P_j$, $P_{j+k}$, and $Q_{j+n} = q_j \cap q_{j-1}$ lie on one line. It remains to note that $P_j P_{j+k} = P_{j+n} P_{j+n+1}$, and we get the desired condition of $Q_j$ lying on $P_j P_{j+1}$ for $i \leq j \leq i+k-1$ and $j = i + 2n-k$.
\end{proof}

The time evolution is going to consist of a single urban renewal at the square face with vertices $P_i, q_i, P_{i+k}, q_{i-1}$. 

\begin{lemma} \label{lem:spir}
 The urban renewal at the face with vertices $P_i, q_i, P_{i+k}, q_{i-1}$ transforms the whole bipartite graph into an identical graph but with index shift $i \mapsto i+1$.
\end{lemma}

\begin{proof}
 The four new vertices created by the urban renewal correspond to points $P_i P_{i+k} \cap P_{i+1} P_{i+k+1} = P_{i+n+1}$ and $P_i P_{i+k} \cap P_{i+n} = P_{i+n}$, and lines $\langle q_i\cap q_{i-1} \cap q_{i+k} \cap q_{i+k-1} \rangle = Q_{i+n} Q_{i+n+k} = q_{i+n}$ and $\langle q_{i} \cap q_{i-1}, q_{i+n-k}\rangle = q_{i+n-k}$. Out of those $P_{i+n}$ and $q_{i+n-k}$ are already present in the graph, and one uses degree two vertex removal move for them, getting rid of vertices with labels $P_i$ and $q_{i-1}$. The overall effect is exactly of shifting the window of seeds by one, so that $P_i$ and $q_{i-1}$ are lost, while $P_{i+n+1}$ and $q_{i+n}$ are gained. 
 
 One can also check that the shape of the graph remains correct, once the removal of degree two vertices mentioned above is applied. In particular, one can easily see from the figure of urban renewal that edges $q_i P_{i+k}$ and $q_{i+n} P_{i+k-1}$ are missing from the new graph, while the edge $q_{i+n-k} P_{i+n}$ is not missing anymore. This achieves the desired shift $i \mapsto i+1$ in all the conditions describing the initial toric bipartite graph. 
\end{proof}

Thus, performing uban renewals produces further and further elements of pentagram spirals $P$ and $Q$. We are ready to state our main dynamic theorem for pentagram spirals: the condition of one spiral being inscribed into the other eventually starts propagating itself. 

\begin{theorem}\label{thm:pentspiral}
Let $P, Q \in \pp_{k,n}$ be a pair of pentagram spirals. If the condition $Q_j \in P_j P_{j+1}$ holds for $2n+1-k$ consecutive values of $j$, then it holds for all values of $j$. 
\end{theorem}

\begin{proof}
 Follows from Theorem \ref{mainThm} and Lemma \ref{lem:spir}. 
\end{proof}

\begin{example}
In the example in Figure \ref{fig:spiral}, it is enough for us to know that $Q_j \in P_j P_{j+1}$ for $1 \leq j \leq 9$, to conclude that this property holds for any $j$.
\end{example}

\section{Example: $Q$-nets}

The following definition is an adaptation of \cite[Definition 2.1]{bobenko2008}. Denote by $\mathbb Z^2_{e}$ the set $\{({i,j}) \in \mathbb Z^2 \mid i+j  \text{ is even}\}$ and by $\mathbb Z^2_{o}$ the set $\{({i,j}) \in \mathbb Z^2 \mid i+j  \text{ is odd}\}$. A map $f: \mathbb Z^2_e \rightarrow \mathbb R^3, ({i,j}) \mapsto f_{i,j}$ is called a {\it {discrete conjugate net}}, or {\it {$Q$-net}} if for any $({i,j}) \in \mathbb Z^2_o$ the four points $f_{i \pm 1, j}, f_{i, j \pm 1}$ are coplanar. Similarly a map $f: \mathbb Z^2_o \rightarrow \mathbb R^3$ is called a {\it {discrete conjugate net}}, or {\it {$Q$-net}} if for any $({i,j}) \in \mathbb Z^2_e$ the four points $f_{i \pm 1, j}, f_{i, j \pm 1}$ are coplanar. 

\begin{remark}
 In \cite{bobenko2008} a more general notion of $Q$-nets is considered for maps $\mathbb Z^m \rightarrow \mathbb R^N$ for any $m$ and $N$. Soon we will need the case $m=N=3$ to state our main theorem. The definition however remains the same: image of any {\it {fundamental quadrilateral}} of $\mathbb Z^m$ has to be coplanar. In our definition we take advantage of the fact that $\mathbb Z^2_{e}$ and $\mathbb Z^2_{o}$ are themselves isomorphic to $\mathbb Z^2$. 
\end{remark}

While the origins of $Q$-nets can be traced to 1930s \cite{sauer1931}, their modern definition is due to Doliwa and Santini \cite{doliwa1997, doliwa2000}.


We define time evolution using {\it {Laplace transform}} of $Q$-nets, see \cite[Exercise 2.13]{bobenko2008}. Namely, let 
$$f'_{i,j} = \langle f_{i-1,j}, f_{i,j-1} \rangle \cap  \langle f_{i+1,j}, f_{i,j+1} \rangle.$$ We make this definition for each $({i,j})$ {\it {not}} in the domain of $f$. This way as time evolution proceeds we keep switching between the $Q$-nets defined on $\mathbb Z^2_e$ and on $\mathbb Z^2_o$. Laplace transform is well-defined since according to the definition of $Q$-nets we are taking two lines in the same projective plane, thus they intersect. It is known that Laplace transform of a $Q$-net is a $Q$-net, see before-mentioned \cite[Exercise 2.13]{bobenko2008}.

One can also consider a dual object called $Q^*$-net, given by a map $$G: \mathbb Z^2_e \rightarrow \text {planes in }\mathbb R^3 \text{ or } G: \mathbb Z^2_o \rightarrow \text {planes in }\mathbb R^3,$$ subject to the condition that the four planes $g_{i \pm 1, j}, g_{i, j \pm 1}$ intersect at one point for each appropriate $({i,j})$. As it is already noticed in \cite{bobenko2008}, $Q$-nets and $Q^*$-nets are essentially the same object, as one can transition between the two ways to keep the data as follows. Given $G: \mathbb Z^2_e \rightarrow \text {planes in }\mathbb R^3$, define $g: \mathbb Z^2_o \rightarrow \mathbb R^3$ via 
$$g_{i,j} = G_{i-1,j} \cap G_{i,j-1} \cap G_{i+1,j} \cap G_{i,j+1}.$$ 
This is well-defined according to the definition of $Q^*$-nets. Furthermore, it is a $Q$-net, as $g_{i,j}, g_{i+2,j}, g_{i+1,j+1}, g_{i+1,j-1}$ all lie on $G_{i+1,j}$. 

We define time evolution on $Q^*$-nets via dual Laplace transform 
$$G'_{i,j} = \langle G_{i-1,j} \cap G_{i,j+1},  G_{i+1,j} \cap G_{i,j-1} \rangle.$$
This is well-defined as we are taking a plane passing through two lines that intersect at a point, according to the definition of $Q^*$-nets. 

\begin{lemma}
In terms of data $g_{i,j}$ of a $Q^*$-net $G$ this is equivalent to the above evolution 
$$g'_{i,j} = \langle g_{i-1,j}, g_{i,j-1} \rangle \cap  \langle g_{i+1,j}, g_{i,j+1} \rangle.$$
\end{lemma}

\begin{proof}
We have $$g'_{i,j} = G'_{i-1,j} \cap G'_{i,j-1} \cap G'_{i+1,j} \cap G'_{i,j+1}.$$
We know that 
$$G'_{i-1,j} = \langle G_{i-2,j} \cap G_{i-1,j+1},  G_{i,j} \cap G_{i-1,j-1} \rangle,$$
$$G'_{i,j-1} = \langle G_{i-1,j-1} \cap G_{i,j},  G_{i+1,j-1} \cap G_{i,j-2} \rangle,$$
$$G'_{i+1,j} = \langle G_{i,j} \cap G_{i+1,j+1},  G_{i+2,j} \cap G_{i+1,j-1} \rangle,$$
$$G'_{i,j+1} = \langle G_{i-1,j+1} \cap G_{i,j+2},  G_{i+1,j+1} \cap G_{i,j} \rangle.$$

One can easily see that $$\langle g_{i-1,j}, g_{i,j-1} \rangle = G_{i,j} \cap G_{i-1,j-1}, \text{ and }$$  $$\langle g_{i+1,j}, g_{i,j+1} \rangle = G_{i,j} \cap G_{i+1,j+1}.$$ Thus each of the lines $\langle g_{i-1,j}, g_{i,j-1} \rangle$,  $\langle g_{i+1,j}, g_{i,j+1} \rangle$ lies on one of the four planes we are intersecting, and thus their intersection points coincide. \end{proof}

Let $f: \mathbb Z^2_e \rightarrow \mathbb R^3$ and $g: \mathbb Z^2_e \rightarrow \mathbb R^3$ be two $Q$-nets. 
Following \cite[Definition 2.7]{bobenko2008} we say that they are {\it {$F$-transforms}} of each other if for each $({i,j})$ and for each choice of the sign the points $f_{i,j}, g_{i,j}, f_{i+1,j \pm 1}, g_{i+1,j \pm 1}$ are coplanar.  

Now assume we have two $Q$-nets $f: \mathbb Z^2_e \rightarrow \mathbb R^3$ and $G: \mathbb Z^2_o \rightarrow \text {planes in }\mathbb R^3$ given by points and planes respectively. Consider a bipartite graph on torus with vertices labelled by $f_{i,j}$ for $({i,j}) \in \mathbb Z^2_e$ and $G_{i,j}$ for $({i,j}) \in \mathbb Z^2_o$, see Figure \ref{fig:qnet} (as usual, ignore dashed edges and vertices with primed labels).

\begin{proposition} \label{prop:qnet}
The conditions (V) and (F) for this bipartite graph mean that $f_{i,j}$ and  $G_{i,j}$ are $Q$-nets, and that one of the $Q$-nets is an $F$-transform of the other. 
\end{proposition}

\begin{proof}
 It is clear that the circuit condition (V) is exactly the defining condition of a $Q$-net (or of a $Q^*$-net). 
 
 For the second statement, consider coherence of the face with vertex labels $f_{i,j}, f_{i+1,j \pm 1}, G_{i+1,j }, G_{i, j \pm 1}$. It is easy to see that the intersection of the planes $G_{i+1,j } \cap G_{i, j \pm 1}$ coincides with the line $\langle g_{i,j}, g_{i+1,j \pm 1} \rangle$. Thus, this coherence means that lines $\langle f_{i,j}, f_{i+1,j \pm 1} \rangle$ and $\langle g_{i,j}, g_{i+1,j \pm 1} \rangle$ intersect, which is the defining property of an $F$-transform. 
\end{proof}

\begin{figure}[!htbp]
    \centering
    \includegraphics[width=\textwidth]{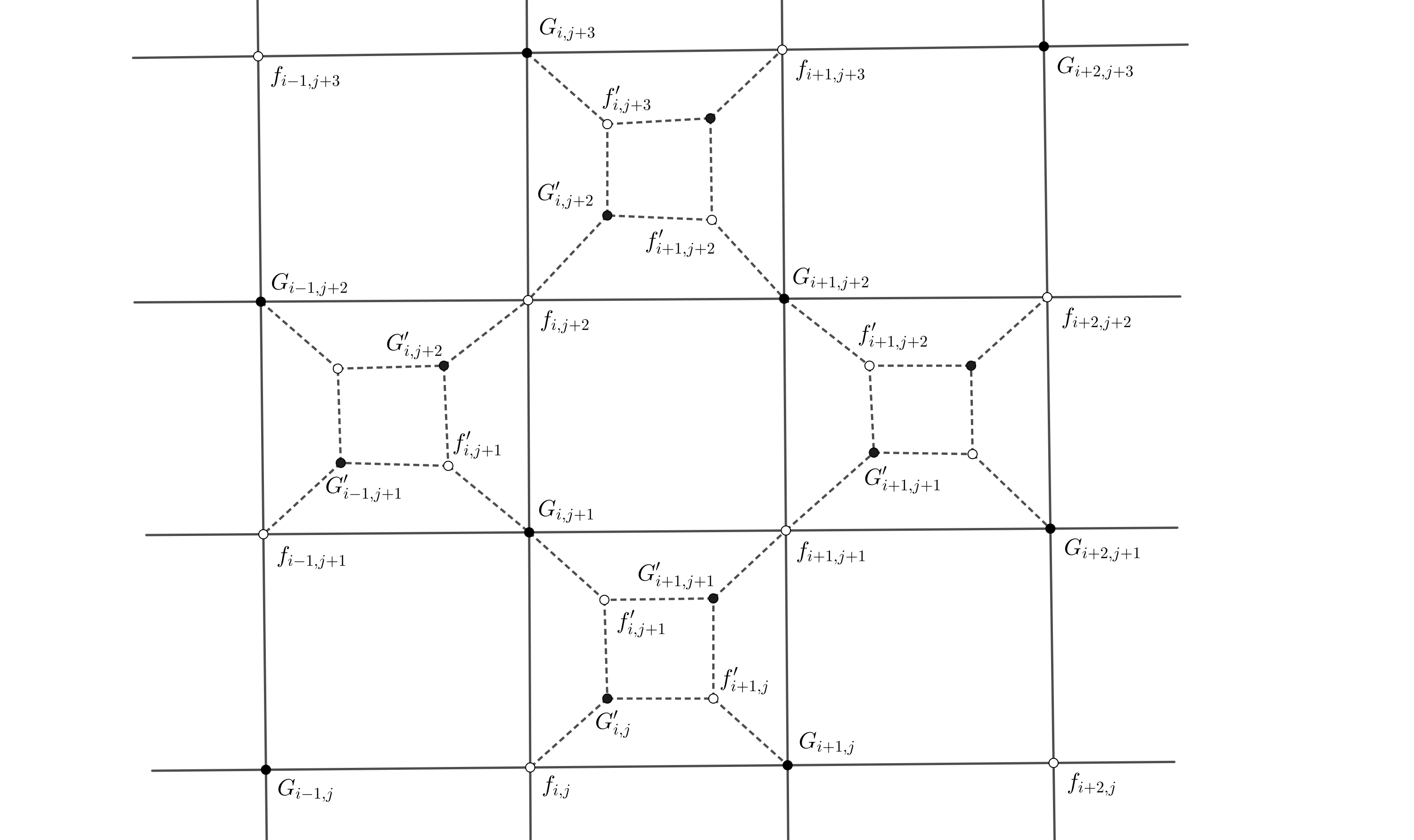} 
    \caption{A fragment of a planar graph associated with two $Q$-nets.}
    \label{fig:qnet}
\end{figure}

\begin{theorem} \label{thm:qnet}
 Urban renewal at faces $f_{i,j}, G_{i,j+1}, f_{i+1,j+1}, G_{i+1,j}$ at odd times and at faces $f_{i,j}, G_{i,j-1}, f_{i+1,j-1}, G_{i+1,j}$ at even times produces exactly the time evolution of both $f$ and $g$ given by Laplace trransform, as defined above. 
\end{theorem}

\begin{proof}
 Follows directly from the rules of urban renewal. For example, the new  $f'_{i+1,j}$ is given by $\langle f_{i,j}, f_{i+1,j+1} \rangle \cap \langle f_{i-1,j+1}, f_{i,j+2} \rangle$, etc. 
\end{proof}

\begin{theorem} \label{thm:laplace}
 Let $h:  \mathbb Z^3 \rightarrow \mathbb R^3$ be a $3$-dimensional $Q$-net. Then the Laplace transform of $h$ is also a $Q$-net. 
\end{theorem}

\begin{proof}
 Consider two consecutive layers of $h_{i,j,k}$ obtained by fixing two consecutive values of the thrird coordinate: $k, k+1$. Let $f_{i,j} = h_{i,j,k}$ and let $g_{i,j} = h_{i,j,k+1}$. Then the fact that $h$ is a $Q$-net is equivalent to each of $f,g$ being a ($2$-dimensional) $Q$-net, and also $f$ and $g$ being $F$-transforms of each other. From Theorem \ref{mainThm} and Theorem \ref{thm:qnet} it follows immediately that those conditions are preserved by Laplace transform. 
\end{proof}

This theorem is not new, in a slightly different language it has appeared as \cite[Theorem 12]{bobenko2007} and \cite[Theorem 2.14]{bobenko2008}. The value of our approach is the automatic generation of the theorem by urban renewal on coherent double circuit configurations.

\section{The reconstruction problem and the spectral curve}\label{sec:rec}

Let $\Sigma$ be a closed orientable surface, and $\Gamma \subset \Sigma$  be a bipartite graph which partitions $\Sigma$ into topological disks. Suppose also that we are given a {$d$-dimensional circuit configuration} on $\Gamma$  (i.e., an assignment of a point in $\P^d$ to every white vertex of $\Gamma$ such that the points assigned to the neighbors of any black vertex form a circuit, see Definition \ref{def:cc}). We will refer to this configuration as the \emph{white data}. Given such white data, we want to know whether we can upgrade it to a coherent double circuit configuration. To that end, we need to supplement the white data with  \emph{black data}, i.e., an assignment of a hyperplane in $\P^d$ to every black vertex of~$\Gamma$ such that together the white and black data form a coherent double circuit configuration.

Suppose that the graph $\Gamma$ has $e$ edges and $f$ faces. Additionally, for the rest of this section we will assume that $\Gamma$ has equal number $k$ of white and black vertices. This condition is perfectly natural in the context of the dimer model, since otherwise $\Gamma$  has no dimer covers.

 Note that the space of possible assignments of hyperplanes to black vertices of $\Gamma$ has dimension $kd$. 
 
 \begin{proposition}\label{prop:dim}
 The coherent double circuit configuration condition imposes  $k(d+2 )- e + f -1$ equations on the black data. \end{proposition}
 \begin{proof}
 For each white vertex, we have the corresponding circuit condition for its neighbors. The condition for $m \leq d+2$ points in $\P^d$ to form a circuit is equivalent to $d+2 - m$ equations. So, using that the sum of degrees of white vertices equals the number of edges, we find that the total number of such equations is
 $k(d+2) - e$.  Additionally, we have $f$ coherence equations. However, by the master theorem (Theorem \ref{mthm}), any $f-1$ of such equations imply the last one, giving the desired formula for the total number of equations. \end{proof}

Fix some white data, and denote by $\mathcal B$ the space of compatible black data (by which we mean that together white and black data they form a coherent double circuit configuration). The space $\mathcal B$ is a quasi-projective algebraic variety.

\begin{proposition}\label{prop:dim2} Assume that the equations from Proposition~\ref{prop:dim} are independent. Then the dimension of the space $\mathcal B$ of black data compatible with the given white data is
$
 1 - \chi
$
where $\chi$ is the Euler characteristic of $\Sigma$. 
\end{proposition}
\begin{proof}
We have $k(d+2 )- e + f -1$ equations on $kd$ parameters, so the dimension is
$
 -2k  + e - f + 1 = 1 - \chi. 
 $
\end{proof}
In particular, if $\Sigma$ is a torus, then the space $\mathcal B$ of black data compatible with  given white data is a {curve} (assuming the equations from Proposition~\ref{prop:dim} are independent). In this section, we will show that this curve is closely related to the \emph{spectral curve} associated with the given black data. Namely, denoting the spectral curve by $\mathcal S$, we always have a natural rational map  $\mathcal F \colon \mathcal B \to \mathcal S$. Furthermore, we will see that in examples this map is birational.

First, let us recall the notion of a spectral curve. In \cite{kenyon2006dimers, GK}, such a curve is associated to an edge-weighted bipartite graph on a torus. Suppose $\Gamma$ is a bipartite graph on a torus whose edges are endowed with non-zero numbers, called in this context \emph{Kasteleyn weights}. Then its universal cover $\widetilde \Gamma$ is a bi-periodic weighted graph in $\R^2$. Denote by $W$ and $B$ the sets of white and black vertices of $\widetilde \Gamma$ respectively, and let $\C^W$, $\C^B$ be the spaces of functions $W \to \C$ and $B \to \C$. 
\begin{definition}
The Kasteleyn operator $\mathcal K \colon \C^W \to \C^B$ is defined by the rule
$$
\mathcal K (f)(v) = \sum_{\mathmakebox[2em][c]{e = (v,w)}}\kappa(e)f(w)
$$
where $v \in B$ is a black vertex of $\widetilde \Gamma$, the sum is taken over all edges $e = (v,w)$ incident to $v$, and $\kappa(e)$ is the Kasteleyn weight of the edge $e$.

\end{definition}
 Further, we have a natural action of $H_1(T^2, \Z) = \pi_1(T^2)$ on $\tilde \Gamma$, and hence a representation on $\C^W$ and $\C^B$. For $ \alpha \in H^1(T^2, \mathbb C^*)$, denote the corresponding weight spaces by $\C^W_\alpha, \C^B_\alpha$. We have
$$
\C^W_\alpha= \{ f \in \C^W_\alpha \mid \gamma \in H_1(T^2, \Z), v \in W \implies f(\gamma \cdot v) = \alpha(\gamma)f   \},
$$
and $\C^B_\alpha$ is defined analogously. Note that any function $f \in \C^W_\alpha$ is uniquely determined by its values on white vertices belonging to the fundamental domain of the action of $H_1(T^2, \Z) = \pi_1(T^2)$ on $\tilde \Gamma$. Therefore, $\dim \C^W_\alpha = k$ . Analogously, $\dim \C^B_\alpha = k$. 
\begin{definition}
The \emph{spectral curve} $\mathcal S$ is  the set of $ \alpha \in H^1(T^2, \mathbb C^*)$ such that the restriction $\mathcal K\vert_{ \C^W_\alpha} \colon  \C^W_\alpha \to  \C^B_\alpha$ is not an isomorphism.
\end{definition}

Choosing suitable bases in $\C^W_\alpha $, $\C^B_\alpha $, one can represent the restriction $\mathcal K\vert_{ \C^W_\alpha} \colon  \C^W_\alpha \to  \C^B_\alpha$ by \emph{magnetically altered Kasteleyn matrix} \cite{kenyon2006dimers}, constructed as follows. The rows and columns of this matrix are labeled by black and white vertices of $\Gamma$ respectively, while the $({i,j})$ entry is given by
$$
\kappa(e) \lambda^{\langle e, \gamma_2 \rangle}\mu^{\langle \gamma_1, e \rangle}
$$
where $e$ is the edge between the black vertex $i$ and white vertex $j$, $\gamma_1, \gamma_2$ are cycles on the torus giving a basis in homology, $\langle e, \gamma_i \rangle$ are intersection numbers computed using a natural (e.g., white to black) orientation of $e$, and $\lambda, \mu$ are coordinates of $ \alpha \in H^1(T^2, \mathbb C^*)$ given by pairing of $\alpha$ with $\gamma_1, \gamma_2$. If there is no edge between the black vertex $i$ and white vertex $j$, then the corresponding entry in equal to $0$. If there are several such edges, then the corresponding entry is given by the sum of expressions as above, one for each edge.

The spectral curve is given by the zero locus of the determinant of this matrix. That determinant is a Laurent polynomial on $H^1(T^2, \mathbb C^*)$, of the form
$$
\sum \kappa(\mathbf{e}) \lambda^{\langle \mathbf{e}, \gamma_2 \rangle} \mu^{\langle \gamma_1, \mathbf{e} \rangle}
$$
where the sum is taken over all dimer covers $e_1,\dots , e_k$ of $\Gamma$, $\mathbf{e} := e_1+ \dots + e_k$, $\kappa(\mathbf{e}):= \kappa(e_1) \cdots \kappa(e_k)$.  In particular, for generic Kasteleyn weights, this determinant is not identically $0$, provided that $\Gamma$ has at least one dimer cover. So, the zero locus of that determinant is indeed a curve.


Now, assume we are given a circuit configuration on a bipartite graph $\Gamma$ on a torus.  That means that we have a point in $\P^d$ assigned to every white vertex $w$ of $\Gamma$ such that the points assigned to the neighbors of any black vertex form a circuit. Lift these points to vectors  $\mathbf{A}(w)$ in $\C^{d+1}$. Then vectors assigned to neighbors of any black vertex are linearly dependent, which means that for every edge $(v,w)$ of $\Gamma$, there exists a number $\kappa(v,w) \in \C$ such that for every black vertex $v$ of $\Gamma$ we have
$$
 \sum_{\mathmakebox[2em][c]{e = (v,w)}}\kappa(e)\mathbf{A}(w) =0
$$
where the sum is taken over all edges $e = (v,w)$ incident to $v$. Moreover, since the set of vectors assigned to neighbors of any black vertex is a circuit, any its proper subset is linearly independent, which implies $\kappa(v,w) \neq 0$ for any edge $(v,w)$. Following \cite{affolter2024vector}, we interpret the coefficients $\kappa(v,w) \in \C^*$ as Kasteleyn weights. This gives a spectral curve assigned to a circuit configuration. The curve is independent on the choice of lifts of points to $\C^{d+1}$, as well as on the choice of the linear equation satisfied by vectors assigned to neighbors of each black vertex.

Thus, for every choice of white data, we get a spectral curve $\mathcal S$. We aim to show that any choice of compatible black data additionally gives a point on the said curve, thus giving rise to a map $\mathcal F \colon \mathcal B \to \mathcal S$ from the space of black data compatible with the given white data to the spectral curve assigned to the same white data. 

First, let us define the notion of a \emph{cohomology class of a coherent tiling}.   Suppose we are given  a tiling, i.e.  
a bipartite graph $\Gamma$ on a surface $\Sigma$, together with an assignment of a point in $\P^d$ to every white vertex and a hyperplane in $\P^d$ to every black vertex. Lift these points and hyperplanes to, respectively, vectors and covectors in $\C^{d+1}$ and assign to each edge of $\Gamma$ the pairing between the vector and covector sitting at its endpoints. Assuming these numbers do not vanish, we can interpret them as a cocycle on $\Gamma$ with coefficients in $\C^*$. This cocyle depends on the choice of lifts of points and hyperplanes to vectors and covectors, but its cohomology class does not. Furthermore, if the given tiling is coherent, then the period of this cohomology class over any contractible cycle is equal to $1$, which means that it belongs to the image of the embedding $H^1(\Sigma, \C^*) \to H^1(\Gamma, \C^*) $. That way, with any coherent tiling we can associate a cohomology class of the corresponding surface. In particular, this applies to any coherent double circuit configuration. 

\begin{proposition}\label{prop:cccdcc}
Assume that we are given a coherent double circuit configuration on a torus. Then the corresponding cohomology class in $H^1(T^2, \C^*)$ lies on the spectral curve associated with the white data.
\end{proposition}\begin{proof}Lift points and hyperplanes of a coherent double circuit configuration to vectors and covectors respectively. For any black vertex $v$, let $\boldsymbol{\ell}(v)$ be the corresponding covector, and for any white vertex $w$, let $\mathbf{A}(w)$ be the associated vector. Then the associated cohomology class $\alpha \in H^1(T^2, \C^*)$ can be represented by the cocycle which associates the number $\langle \boldsymbol{\ell}(v), \mathbf{A}(w) \rangle$ to every edge $(v,w)$. Lift all the objects to the universal cover of the torus. On the universal cover, the cocycle $\langle \boldsymbol{\ell}(v), \mathbf{A}(w) \rangle$ becomes a coboundary, which means there is a function $f \colon W \sqcup B \to \C^*$ on the edge-set of the universal cover such that
\begin{equation}\label{eq:lvvw}
\langle \boldsymbol{\ell}(v), \mathbf{A}(w) \rangle = f(v)^{-1} f(w).\tag{$\star$}
\end{equation}
Furthermore, since this cocycle represents the class $\alpha \in H^1(T^2, \C^*)$, we have $f\vert_W \in  \C^W_\alpha$, $f\vert_B \in  \C^B_\alpha$.

The spectral curve is defined using the Kasteleyn weights $\kappa(v,w)$ associated with the white data. 
For any black vertex $v$ of the universal cover $\widetilde \Gamma$, the weights satisfy
$$
 \sum_{\mathmakebox[2em][c]{e = (v,w)}}\kappa(e)\mathbf{A}(w) = 0.
$$
Pairing both sides with $\boldsymbol{\ell}(v)$, we get
$$
f(v)^{-1}\sum_{\mathmakebox[2em][c]{e = (v,w)}} \kappa(e) f(w)  = 0,
$$
so $f\vert_W \in \mathrm{Ker}\, \mathcal K\vert_{ \C^W_\alpha}$, which means that $\alpha$ belongs to the spectral curve. 
\end{proof}

\begin{question}
Analogously, one can associate a spectral curve to the black data and show that the cohomology class of a coherent double circuit configuration on a torus lies on that curve as well.   Are spectral curves associated with white and black data in a coherent double circuit configuration the same?
\end{question}
From Proposition \ref{prop:cccdcc},  we get  a rational map $\mathcal F$ from from the space $\mathcal B$ of black data compatible with the given white data (which is a curve assuming the equations from Proposition~\ref{prop:dim} are independent) to the spectral curve $\mathcal S$ assigned to the same white data. 

\begin{remark}
The construction of the map  $\mathcal F$, as well as coincidence of dimensions of $\mathcal B$ and $\mathcal S$, carries over to double circuit configurations on more general surfaces. Specifically, for a surface of genus $g$, both dimensions are equal to $2g-1$. For the variety $\mathcal B$ of black data, this is Proposition \ref{prop:dim2} (which holds assuming that the equations from Proposition \ref{prop:dim} are independent), while for the spectral curve $\mathcal S$ (which should be called in this setting the \emph{spectral surface}) this follows from its definition as a hypersurface in the cohomology.\end{remark}

Since $\mathcal F \colon \mathcal B \to \mathcal S$ is a rational map between varieties of the same dimension, it is natural to ask whether it is birational. Below we consider some examples where it is or is not the case.

\begin{problem}
Find general conditions of the graph $\Gamma$ and dimension $d$ which guarantee that the map $\mathcal F$ from black data to the spectral curve is birational for generic white data.
\end{problem}

\begin{example}\label{prop:pentbir}
Let us show that, for $2$-dimensional double circuit configurations on graphs corresponding to pentagram maps $ T_k$, the map $\mathcal F \colon \mathcal B \to \mathcal S$ is birational  for generic white data. In the language of polygons, this means that for a generic polygon $P$, there exists a one-dimensional family of polygons $Q$ such that  $Q$ is inscribed in $P$ and $T_k(Q)$ in $T_k(P)$. This family is parametrized by the spectral curve associated to $P$.

Assume that the spectral curve $\mathcal S$ is irreducible and hence smooth everywhere except for at finitely many points. We will show that any generic smooth point $\alpha \in \mathcal S \subset H^1(T^2, \C^*)$ has a unique preimage under $\mathcal F$. In other words, for any generic smooth point $\alpha \in \mathcal S$, there exists a unique way to update the given white data to a double circuit configuration whose cohomology class is $\alpha$. 

Fix a lift of the white data to vectors and Kasteleyn weights. Since $\alpha$ is a smooth point of the spectral curve, the kernel of the Kasteleyn operator $\mathcal K\vert_{ \C^W_\alpha}$ is one-dimensional. Thus, one can recover the function $f\vert_W \colon W \to \C$ from the proof of Proposition  \ref{prop:cccdcc} uniquely up to a constant factor. For generic white data, this function takes values in $\C^*$. Combining it with an arbitrary non-vanishing function $f\vert_B \in  \C^B_\alpha$, we get a function $f \colon W \sqcup B \to \C^*$. Now, to recover the black data, we need to find a function $\ell \colon W \to \C^3$ such that, for every edge~$(v,w)$, we have equation \eqref{eq:lvvw}. 
This gives four linear equations for each of the vectors $\boldsymbol{\ell}(v)$. Additionally, for any black vertex $v$, we have
$$
\sum_{\mathmakebox[2em][c]{e = (v,w)}} \kappa(e) f(w)  = 0
$$
(because $f\vert_W$ is in the kernel of the Kasteleyn operator)
and 
$$
 \sum_{\mathmakebox[2em][c]{e = (v,w)}}\kappa(e)\mathbf{A}(w) = 0
$$
(by definition of Kasteleyn weights), 
where the sums are taken over all edges incident to $v$. This means that one of the four equations is redundant. At the same time, the remaining equations are independent by the circuit condition for the vectors $\mathbf{A}(w)$. Thus, the covectors $\boldsymbol{\ell}(v)$ are uniquely recovered from the corresponding point on the spectral curve. Together, the vectors $\mathbf{A}(w)$ and covectors $\boldsymbol{\ell}(v)$ give rise to a double circuit configuration. The circuit condition for the covectors $\boldsymbol{\ell}(v)$ is trivial because the degree of all vertices is equal to $4$, while the coherence condition is a consequence of \eqref{eq:lvvw}. This double circuit configuration on the universal cover $\tilde \Gamma$ descends to $\Gamma$ since $f\vert_W, f\vert_B$ satisfy the same quasi-periodicity condition: $f\vert_W \in  \C^W_\alpha$, $f\vert_B \in  \C^B_\alpha$. Choosing a different extension $f\vert_B$ of $f\vert_W$ to black vertices, we get proportional covectors $\boldsymbol{\ell}(v)$ and hence the same double circuit configuration.

\end{example}

The same argument works for any graph such that all vertices have degree $d+2$, where $d$ is the dimension of the underlying double circuit configuration (note that this is the maximal possible degree, since $d+3$ points in $\P^d$ cannot form a circuit). In the case of vertices of smaller degree,  equations \eqref{eq:lvvw} on covectors $\boldsymbol{\ell}(v)$ become underdetermined and should be supplemented by additional equations coming from circuit conditions on $\boldsymbol{\ell}(v)$. The latter equations are nonlinear, which makes the unique solvability issue more subtle. 
%

\begin{example}
Consider a $2$-dimensional circuit configuration on a graph on a torus obtained from a square grid by removing one edge. Let us show that in this setting any point on the spectral curve has either zero or infinitely many preimages under $\mathcal F$, which means that $\mathcal F$ is not birational. Let $v$ be the degree $3$ black vertex. Fixing a point on the spectral curve gives three linear equations on the covector $\boldsymbol{\ell}(v)$, of the form $\langle \boldsymbol{\ell}(v), \mathbf{A} (w_i)\rangle = \dots$, where $\mathbf{A}(w_i) $ are vectors assigned to neighbors $w_i$ of $v$. As explained above, one of these equations is redundant, so there are either zero or infinitely many solutions $\boldsymbol{\ell}(v)$. There are no further equations on $\boldsymbol{\ell}(v)$ coming from cicuit conditions, because all neighbors of $v$ are of degree $4$. 
%
\end{example}

\begin{figure}[t]
    \centering
    \includegraphics[width=0.65\textwidth]{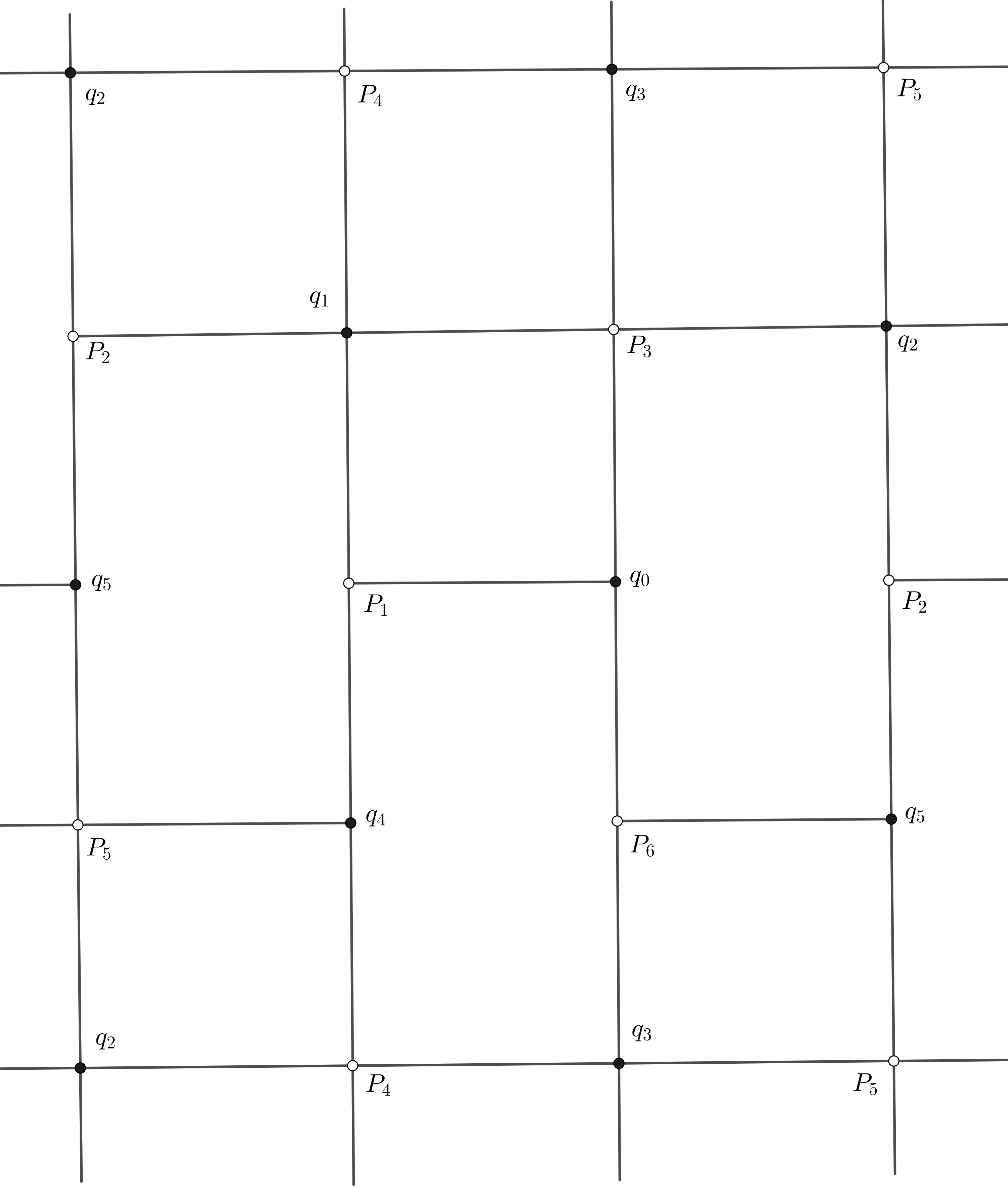} 
    \caption{A toric bipartite graph for spirals in $\pp_{2,5}$.}
    \label{fig:spiralgraph2}
\end{figure}

\begin{example}
Figure \ref{fig:spiralgraph2} shows a toric bipartite graph for spirals in $\pp_{2,5}$ (points with identical labels are identified). Here the map $\mathcal F \colon \mathcal B \to \mathcal S$ is birational in dimension $2$ despite the presence of degree $3$  vertices. The lines $q_i$ are computed from the point on the spectral curve as follows. First, one readily finds $q_1, q_2, q_3$ since they are located at vertices of degree $4$. Next, we can find $q_5$ as follows. Since the corresponding vertex is of degree $3$, there are two linear equations of the form $\langle \boldsymbol{q}_5, \dots \rangle = \dots$ on the lift $\boldsymbol{q}_5$ of ${q}_5$. Additionally, the vertex with label $P_2$ dictates that $\boldsymbol{q}_5$ is a linear combination of $\boldsymbol{q}_1$ and $\boldsymbol{q}_2$ which are already known. Together, we get three linear equations which uniquely determine $ {q}_5$ (the equations are independent for a generic choice of $P_1, \dots, P_6$). After that, we use the equation determined by the vertex with label $P_6$ to find $q_0$, and then finally the equation determined by the vertex with label $P_1$ to find $q_4$. 
\end{example}

Another way of finding the lines $q_0, q_4, q_5$ is by bringing the degrees of the corresponding vertices up to $4$ by using appropriate sequences of local moves. For example, perform an urban renewal at the face labeled $q_3P_5q_5P_6$ and assign points to newly created white vertices by using the urban renewal tule for circuit configurations. Then, contracting the white vertex labeled $P_6$, we get a graph where $q_0$ is at a degree $4$ vertex, which allows one to find it from the point on the spectral curve by using equations $\langle \boldsymbol{q}_0, \dots \rangle = \dots$ corresponding to edges incident to the vertex labeled $q_0$.

More generally, given a toric bipartite graph, it might be possible to invert the map $\mathcal F \colon \mathcal B \to \mathcal S$ in dimension $d$ by finding, for every white vertex, a sequence of moves which raises its degree to $d+2$. Experimental evidence suggests that for graphs that are \emph{minimal} in the sense of \cite{GK}, it is always possible to raise the degree of any given vertex to be equal to the number of vertices of the corresponding Newton polygon. 

\begin{conjecture}
Suppose we are given a minimal bipartite graph on a torus whose Newton polygon has $d+2$ vertices. Then the map $\mathcal F \colon \mathcal B \to \mathcal S$ is birational for generic white data in dimension $d$. 
\end{conjecture}

\bibliographystyle{plain}
\bibliography{inc.bib}

\end{document}